\documentclass[12pt]{amsart}
\usepackage{amssymb, a4wide, mathdots,url,hyperref, shuffle}
\usepackage{bbm}
\usepackage{tikz}
\usepackage{ytableau}
\usepackage{pb-diagram}

\newcommand{\Hom}{\operatorname{Hom}}
\renewcommand{\subset}{\subseteq}

\newcommand{\cusp}{{\operatorname{cusp}}}

\newtheorem{theorem}{Theorem}[section]
\newtheorem{lemma}[theorem]{Lemma}
\newtheorem{proposition}[theorem]{Proposition}
\newtheorem{remark}[theorem]{Remark}

\newtheorem{corollary}[theorem]{Corollary}

\date{\today}

\newcommand{\gotM}{\mathfrak{m}}

\DeclareMathOperator{\irr}{Irr}

\newcommand{\pairs}{\mathfrak{A}}

\newcommand{\smlr}{\ll}

\newcommand{\supp}{\operatorname{supp}}

\newcommand{\Seg}{\operatorname{Seg}}

\newcommand{\m}{\mathfrak{m}}
\newcommand{\n}{\mathfrak{n}}
\newcommand{\la}{\mathfrak{l}}
\newcommand{\Vien}{\mathcal{K}}
\newcommand{\Vein}{\mathcal{K'}}
\newcommand{\depth}{\mathfrak{d}}
\newcommand{\Mult}{\mathfrak{M}}

\renewcommand{\subset}{\subseteq}

\newcommand{\RSK}{\mathcal{RSK}}

\newcommand{\Lad}{\operatorname{Lad}}
\newcommand{\lshft}[1]{\overset{\leftarrow}{#1}\vphantom{#1}}               

\newcommand{\wt}{\mathrm{wt}}

\newcommand{\nmod}{\mathrm{mod}}
\newcommand{\fgt}{\mathrm{fgt}}

\newcommand{\nmodi}{\mathrm{Mod}}

\newcommand{\gmod}{\mathrm{gmod}}

\newcommand{\KR}{\mathrm{KR}}

\newcommand{\rres}{\mathrm{Res}}
\newcommand{\iind}{\mathrm{Ind}}
\newcommand{\coiind}{\mathrm{coInd}}
\newcommand{\girr}{\mathrm{gIrr}}

\newcommand{\seg}{\mathrm{Seg}}

\begin{document}

\title[Quiver Hecke algebras and RSK]{Simple modules for quiver Hecke algebras and the Robinson--Schensted--Knuth correspondence }

\begin{abstract}
  We formalize some known categorical equivalences to give a rigorous treatment of smooth representations of $p$-adic general linear groups, as ungraded modules over quiver Hecke algebras of type $A$.
  
  Graded variants of RSK-standard modules are constructed for quiver Hecke algebras. Exporting recent results from the $p$-adic setting, we describe an effective method for construction and classification of all simple modules as quotients of modules induced from maximal homogenous data.
  
It is established that the products involved in the RSK construction fit the Kashiwara-Kim notion of normal sequences of real modules. We deduce that RSK-standard modules have simple heads, devise a formula for the shift of grading between RSK-standard and simple self-dual modules, and establish properties of their decomposition matrix, thus confirming expectations for $p$-adic groups raised in a work of the author with Lapid.

Subsequent work will exhibit how the presently introduced RSK construction generalizes the much-studied Specht construction, when inflated from cyclotomic quotient algebras. 
\end{abstract}

\author{Maxim Gurevich}
\address{Department of Mathematics, Technion -- Israel Institute of Technology, Haifa, Israel.}
\email{maxg@technion.ac.il}

\maketitle

\section{Introduction}

This work stands at a crossroads of two domains in representation theory. One deals with smooth complex representations of a family of locally compact groups $GL_n(F)$, where $F$ is a $p$-adic field. The other studies modules over quiver Hecke algebras (also known as Khovanov-Lauda-Rouquier algebras) associated to Cartan data of type $A_N$. By adopting a categorical point of view and exploiting known equivalences between those settings, we are able to answer recently posed questions about the nature of some finite-length representations on both sides.

The representation theory of $p$-adic general linear groups plays a major role in the celebrated Langlands program. 
An early step towards the program's local goals was the ground-breaking Zelevinsky classification \cite{Zel}, which suggested a combinatorial construction of all irreducible representations of $GL_n(F)$ building upon given supercuspidal data. While the original local Langlands reciprocity for these groups is by now well-understood, an effective description of possible indecomposable objects in this non-semi-simple category remains an intriguing task that reaches beyond the traditional scope of the Langlands program (see \cite{lm-inven} for a possible visionary direction). Our work continues an exploration of a new method suggested in \cite{gur-lap} for a systematic construction of classes of representations of interest.

Quiver Hecke algebras were introduced in \cite{KLR1,KLR2,rouq}, as a means of categorifying quantum groups. Namely, a choice of a simple Lie algbera $\mathfrak{g}$ gives rise to a family of associative algebras $\{R(\beta)\}_{\beta}$, whose finite-dimensional module categories, when put together, provide a monoidal abelian category. The Grothendieck ring of the resulting category is then identified with an integral form of the positive part of the quantum group $U_q(\mathfrak{g})$.

The quiver Hecke construction stood as an algebraic explication of previously known geometric categorifications, such as the one in \cite{lusztig-quiver}, for quantum groups (see \cite{vv} for details). In particular, questions on the multiplicative structure of $U_q(\mathfrak{g})^+$ relative to its dual canonical basis are lifted to questions on the monoidal structure of explicitly defined categories. In recent years, this approach lead, for example (see \cite{kkko-mon}), to an improved understanding of such structures through the axiomatic framework of cluster algebras.

The two areas of study described above become highly interrelated, when the Lie type $\mathfrak{g} = \mathfrak{sl}_N$ (or, more conveniently $\mathfrak{sl}_\infty$) is taken. A middle link between the two settings is the representation theory of affine Hecke algebras. It is classically established, and more modernly proved (\cite{bk-book,heier-cat}) for our cases, that the representation theory of a suitable affine Hecke algebra should mimic smooth representations of $p$-adic group with a fixed supercuspidal data (a Bernstein block). We explicate the known results for our groups of interest into a functor, which we call the Bernstein equivalence.

On the other hand, from the outset of the theory of quiver Hecke algebras, it was known that in type $A$ their representation theory largely coincides with that of affine Hecke algebras. In fact, it was shown \cite{brun-kles} that taking natural (cyclotomic) quotients of both kinds of algebras, results in isomorphic algebras. Following \cite{MR2908731}, we take a functorial view on this identification, which we call the Rouquier equivalence.

An inherent feature of quiver Hecke algebras is that they are graded. In fact, the graded structure on their modules stands as a categorification of the algebraic quantization parameter in the quantum group. Namely, the integral form of $U_q(\mathfrak{g})^+$ comes with a $\mathbb{Z}[q,q^{-1}]$-algebra structure, which is viewed as shifts of grading on modules in the Grothendieck ring realization. This structure remains practically hidden in the setting of representations of $p$-adic groups.

In the first part of this work we formalize both Bernstein and Rouquier functors, and compose them together (Theorem \ref{thm:fullequi}) into an explicit equivalence of monoidal categories. This formalism provides rigorous tools for the conceptual treatment of finite-dimensional representation theory of quiver Hecke algebras of type $A$, as a gradation, or quantization, of the smooth finite-length representation theory of $GL_n(F)$.

It should be noted that throughout this article all base fields are assumed to be complex, while $q\in \mathbb{C}$, whenever relevant, is not a root of unity.

\subsection{Irreducible representations}
The quantization point of view fits well with the accumulated knowledge on the collection of irreducible objects in categories of both kinds. The various bases for $U_q(\mathfrak{sl_\infty})^+$ may be naturally parameterized by representations of the $A_\infty$-quiver (e.g. \cite{lus-canonical,LNT}), which are set in a clear bijection with the multisegments used in the Zelevinsky classification. Since simple modules of type $A$ quiver Hecke algebras categorify the dual canonical basis, the same parametrization should hold in that setting.

Indeed, such a classification of simple modules, generalizing Zelevinsky, was conducted directly by Kleshchev-Ram in \cite{kr2}. Moreover, they constructed families of finite-length proper standard modules (later generalized in \cite{mcn-pbw,mcn-bk}) categorifying PBW bases for the quantum group, and, from our point of view, serving as a graded version of the Langlands/Zelevinsky standard representations for $p$-adic groups.

One well-understood family of simple modules for quiver Hecke algebras is that of \textit{homogeneous} modules, that is, modules concentrated at one degree of their grading. These were treated in \cite{kr1} and were classified for type $A$ algebras. Unsurprisingly, the $p$-adic groups literature was independently familiar with the class of irreducible representations appearing when transferring results through the above mentioned equivalences. These were coined as ladder representations \cite{LM}, after the shape of the Zelevinsky multisegments describing them.

On the $p$-adic side, the class of ladder representations is known to be especially accessible (e.g. \cite{LapidKret}), while frequently appearing in applications of the theory (such as in the study of harmonic analysis, with the building blocks of the unitary spectrum of $GL_n(F)$ being composed of specific (Speh) ladder representations \cite{Tad}).

\subsection{Graded RSK-standard modules}
Exporting notions developed in \cite{gur-lap}, we may now ask: Is there an efficient, or insightful, method of constructing all simple modules for quiver Hecke algebras using the class of homogeneous simple modules?

This question relates to a new model suggested in \cite{gur-lap} for the construction of all irreducible representations of $GL_n(F)$, which is based on the combinatorial Robinson--Schensted--Knuth correspondence. Let us briefly describe it here in the language of quiver Hecke algebra modules, while deferring most details to Section \ref{sec:rsk}.

Given a simple module $L$ of a quiver Hecke algebra $R(\beta)$ of type $A_N$, the Kleshchev-Ram classification (up to some choices) provides a multisegment $\m$ parameterizing $L$. While $\m$ should be viewed as a positive element of the root lattice of $\mathfrak{sl}_\infty$, for sake of introduction let us treat it as a multi-set of pairs of integers:
\[
\m = \{(a_1,b_1),\ldots, (a_n,b_n)\}\;,
\]
with $a_i\leq b_i$, for all $i$.

Applying the RSK algorithm (e.g. \cite{MR0272654}) on this multi-set, one produces a pair of semistandard Young tableaux $(P_\m, Q_\m)$ of equal shapes. Alternatively, the information encoded in the resulting bi-tableau may be written as a tuple
\[
\RSK(\m) = (\la_1,\ldots,\la_\omega)\;,
\]
so that each $\la_i$ is a (ladder) multisegment itself, read out of the $i$-th pair of rows in $(P_\m, Q_\m)$.

The simple modules $\Xi_1,\ldots,\Xi_\omega$ corresponding to $(\la_1,\ldots,\la_\omega)$ under the Kleshchev-Ram classification now become homogenous representations of corresponding quiver Hecke algebras.


With these data in place, we define the \textit{RSK-standard module} associated with the multisegment $\m$ to be the finite-dimensional $R(\beta)$-module
\[
\Gamma(\m) = \Xi_1\circ\cdots \circ \Xi_\omega \langle -d(\m) \rangle \;,
\]
given by the convolution product, where $\langle -d(\m)\rangle$ denotes a suitable shift of grading.

Passing through our composed functor of the Bernstein and Rouquier equivalences, we may now state a quiver Hecke algebra version of the main result of \cite{gur-lap}.

\begin{theorem}\label{thm:intro1}[Theorem \ref{thm:gurlap}]
Up to a shift of grading, the simple module $L$ classified by the multisegment $\m$ is isomorphic to a quotient module of the RSK-standard module $\Gamma(\m)$.
\end{theorem}


We note that the convolution product constructs $\Gamma(\m)$ as induction from a simple homogeneous module of a parabolic subalgebra (see Section \ref{sec:resind}) $P = R(\beta_1,\ldots,\beta_\omega)< R(\beta)$. It follows from \cite[Section 4]{me-decomp} (again functorially transferring results to a graded version) that $P$ is the maximal parabolic subalgebra, for which $L$ (the simple module classified by $\m$) may occur as a subquotient of a module induced from a homogeneous module of $P$. In other words, there may not be a product of less than $\omega$ simple homogenous terms that will give rise to a construction of $L$.

In this sense, a RSK-standard module may be thought as the module realizing the ``shortest distance" of the class of homogenous modules to a given simple module.



\subsection{Main results}

Several conjectures were raised in \cite{gur-lap} regarding the nature of the RSK-standard modules in the $p$-adic setting, suggesting further favorable properties of the new construction. In this work, we apply the gained advantage of transition into the graded setting to prove some of these expectations.

First, Theorem \ref{thm:intro1} raises the question of whether the constructed irreducible quotient of $\Gamma(\m)$ is unique, i.e., whether RSK-standard modules have simple heads.

In Theorem \ref{thm:mmain} we positively answer the stronger graded variant of this question, thus making the structure of RSK-standard modules considerably more transparent. We prove that the \textit{graded multiplicity} of the simple quotient $L_\m$ of $\Gamma(\m)$ is $1$. Moreover, the graded multiplicity of any other subquotient of $\Gamma(\m)$ must consist of positive degrees.

The resulting picture is a familiar trait seen in transition matrices between PBW-bases and canonical bases in quantum groups. Hence, the role of RSK-standard modules is further established as a fitting alternative to the proper standard modules categorifying PBW-bases.

Using our functorial gateway through its other end, we settle in Corollary \ref{cor:final} some analogous originally posed questions on RSK-standard modules for the context of representations of $p$-adic groups.

Finally, we reach an appealing formula (Theorem \ref{thm:deg}) for the constant $d(\m)$ used to normalize the grading shift on a RSK-standard module. 

Let us draw attention to the similarity of this formula to that of the dimension of a Springer fiber for the symmetric group \cite{spal-dim} parameterized by the shape of the bi-tableau of $\RSK(\m)$. We suggest a further search for possible relations of the RSK construction to a geometric Springer-type approach.

\subsection{Methods}
The key tool used in the proof of Theorem \ref{thm:mmain} is the concept of \textit{normal sequences}. 

The theory developed in the works of Kang-Kashiwata-Kim-Oh \cite{kkko0,kkko-mon} has put the notion of \textit{real} simple modules on the front lines of study: Those are simple modules $L$ of quiver Hecke algebras, for which $L\circ L$ remains simple. The analogous notion was investigated in depth in the $p$-adic setting \cite{LM3}, under the name of \textit{square-irreducible} representations.

For a real simple module $L$ and any simple module $M$ of the suitable quiver Hecke algebra, the product $L\circ M$ has a simple head that is given by $H\langle h \rangle$, where $H$ is self-dual. The numeric invariant $\widetilde{\Lambda}(L,M) = -h$ is a useful outcome of the graded setting approach to our set of problems.

Kashiwara-Kim \cite{kk19} have called a sequence of real simple modules $L_1,\ldots,L_k$ normal, if it satisfies certain compatibility properties, that can be easily stated in terms of the $\widetilde{\Lambda}$-invariant. 

In the case of a normal sequence, the product $L_1\circ\cdots\circ L_k$ has a simple head, among further favorable properties that are established in Proposition \ref{prop:mult-one}.

For a multisegment $\m$, its RSK transform $\RSK(\m) =  (\la_1,\ldots,\la_\omega)$ is defined inductively, so that $\RSK(\widehat{\m}) =  (\la_2,\ldots,\la_\omega)$, for another multisegment $\widehat{\m}$. Now, given the homogeneous (real) simple module $L=\Xi_1$ corresponding to $\la_1$, and the simple head $M$ of $\Gamma(\widehat{\m})$, we produce a combinatorial formula (Corollary \ref{cor:maincor}) for the number $\widetilde{\Lambda}(L,M)$.

It then follows that the sequence of homogenous modules $\Xi_1,\ldots,\Xi_\omega$ used to define $\Gamma(\m)$ is normal. 

\subsection{Links with the Specht construction}

In a subsequent work under preparation, we will show that the RSK construction may be viewed as a generalization of the Specht construction for cyclotomic Hecke algebras \cite{djm}, and its graded variant \cite{bkw}.

Specifically, the classification of simple modules (\cite{ariki-class,groj}) of the cyclotomic quotients of affine Hecke algebras and the analogous quotients of quiver Hecke algerbas, often follows the lines of the Zelevinsky approach, with a reducible Specht module standing in place of a (proper) standard module. When this picture is inflated to the (full) quiver Hecke algebra representation category, a comparison between the Kleshchev-Ram/Zelevinsky construction and the Specht construction begs to be performed.

Indeed, works of Vazirani \cite{MR1923974} and Kang-Park \cite{MR2811321} have studied similar questions. Their results can now be explained through the lens of the RSK construction and our results, since inflated Specht modules turn out to be special cases of RSK-standard modules.

Moreover, our Theorem \ref{thm:mmain} may be specialized to the context of graded multiplicities treated in \cite{bk-decomp}, while the formula of Theorem \ref{thm:deg} is consistent in the special Specht cases with the shift formula of \cite[Theorem 8.2]{univ-specht}, for inflated graded Specht modules.

\subsection{Structure}

We first survey in Section \ref{sec:quiver} the basics of the (necessary cases of the) representation theory of quiver Hecke algebras of type $A$. Specifically we recall the Kleshchev-Ram approach to classification of simple modules.

Section \ref{sec:padic} surveys the smooth finite-length representation theory of $p$-adic $GL_n$, while putting together the categorical bridges needed for comparison of our two settings. 

Using the established gateway between settings, we single out several curious classes of simple modules in Section \ref{sec:special}.

In Section \ref{sec:rsk}, we recall the RSK construction of \cite{gur-lap}, export it into the quiver Hecke algebra domain and develop some of its properties.

The degree computations of Section \ref{sec:deg-comp} is the technical heart of this work. Its aim is to show that the $\widetilde{\Lambda}$-invariant is well-behaved relative to the RSK inductive process.

Finally, we develop some properties of normal sequences in Section \ref{sec:normal} and apply them on the RSK construction.

\subsection{Acknowledgements}
I would like to thank Inna Entova-Aizenbud, Alexander Kleshchev, Erez Lapid, Bernard Leclerc, Alberto M\'{i}nguez, Mark Shimozono and Liron Speyer for the immense help provided by sharing their expertise on aspects of this work.

Special thanks are also due to Dvir Mor and Rida Saabna for supplying initial computational evidence for the validity of the formula in Theorem \ref{thm:deg}, during their participation in Technion's Undergraduate Research Projects week in Summer 2020.

This research is supported by the Israel Science Foundation (grant No. 737/20).

\section{Quiver Hecke algebras}\label{sec:quiver}

Let us recall the basics of the representation theory of quiver Hecke algebras. We will largely follow the standard conventions as in \cite{kr2,brun-kles,MR2822211,kkko-mon}.

Our description will cover algebras of type $A$ solely. More precisely, the general construction of quiver Hecke algebras depends on a choice of Lie-theoretic data as an input. The algebras appearing in our discussion are the ones associated with the data of the Lie algebra $\mathfrak{sl}_\infty$.

We take the Cartan datum $(\mathcal{I},\cdot)$ as a set labelled by integers $\mathcal{I}=(\alpha_i)_{i\in\mathbb{Z}}$ (simple roots), and an integer valued symmetric bilinear form $\alpha,\beta\mapsto (\alpha,\beta)$ on the free abelian group $Q = \mathbb{Z}[\mathcal{I}]$ (root lattice) given by
\[
(\alpha_i,\alpha_j) = \left\{ \begin{array}{ll} 2 & i=j \\ -1 & |i-j|=1 \\ 0 & |i-j|>1 \end{array}\right.\;, \quad \forall i,j\in \mathbb{Z}\;.
\]
We denote the positive cone $Q_+ = \sum_{i\in I} \mathbb{Z}_{\geq 0} \alpha_i\subset Q$. For $\beta_1,\beta_2\in Q_+$, we write $\beta_1\leq \beta_2$ whenever $\beta_2-\beta_1\in Q_+$.

Let $\beta=\sum_{i\in I} c_i \alpha_i\in Q_+$ be fixed. Its height is defined as $|\beta|=\sum_{i\in I}c_i\in \mathbb{Z}$.

We denote the finite set of tuples
\[
\mathcal{I}^\beta = \{ \nu =(\nu_1,\ldots, \nu_{|\beta|})\in \mathcal{I}^{|\beta|}\;:\; \alpha_{\nu_1} + \ldots + \alpha_{\nu_{|\beta|}} = \beta\}\;.
\]

The \textit{quiver Hecke algebra} (or Khovanov-Lauda-Rouquier algebra) related to $\beta$ is defined to be the associative complex algebra $R(\beta)$, which is generated
by $\{\mathfrak{e}(\nu)\}_{\nu \in \mathcal{I}^{ \beta  }}$, $\{y_1,\ldots, y_{|\beta|}\}$, $\{\psi_1, \ldots, \psi_{|\beta|-1}\}$, subject to the relations
\[
\mathfrak{e}(\nu)\mathfrak{e}(\nu') = \left\{\begin{array}{ll} \mathfrak{e}(\nu) & \nu = \nu' \\ 0 & \nu \neq \nu'\end{array}\right.\;,\quad \sum_{\nu\in  \mathcal{I}^{ \beta  }} \mathfrak{e}(\nu) = 1\;,
\]
\[
y_i \mathfrak{e}(\nu) = \mathfrak{e}(\nu) y_i\;,\quad \psi_i \mathfrak{e}(\nu) = \mathfrak{e}(s_i\cdot\nu)\psi_i\;,\;\forall i\;,
\]
\[
y_iy_j = y_iy_j\;,\;\forall i,j\;,\quad \psi_i\psi_j = \psi_j \psi_i\;,\mbox{ for }|i-j|>1\;,
\]
\[
y_j\psi_i = \psi_i y_j\;, \mbox{ for }j\not\in \{i,i+1\}\;,
\]
\[
(y_{i+1}\psi_i - \psi_i y_{i})\mathfrak{e}(\nu) =  \left\{\begin{array}{ll} \mathfrak{e}(\nu) & \nu_i = \nu_{i+1} \\ 0 & \nu_i \neq \nu_{i+1} \end{array}\right.\;,\quad (y_{i}\psi_{i} - \psi_{i} y_{i+1})\mathfrak{e}(\nu) =  \left\{\begin{array}{ll} -\mathfrak{e}(\nu) & \nu_i = \nu_{i+1} \\ 0 & \nu_i \neq \nu_{i+1} \end{array}\right.\;,
\]
\[
(\psi_{i+1}\psi_i\psi_{i+1} - \psi_i\psi_{i+1}\psi_i)\mathfrak{e}(\nu) = \left\{\begin{array}{ll} \mathfrak{e}(\nu) & (\nu_i, \nu_{i+1}, \nu_{i+2}) = (\alpha_t, \alpha_{t+1}, \alpha_t),\mbox{ for }t\in\mathbb{Z} \\ -\mathfrak{e}(\nu) & (\nu_i, \nu_{i+1}, \nu_{i+2}) = (\alpha_t, \alpha_{t-1}, \alpha_t),\mbox{ for }t\in\mathbb{Z} \\ 0 & \mbox{otherwise} \end{array}\right.,
\]
\[
\psi_i^2\mathfrak{e}(\nu) = \left\{\begin{array}{ll} (y_i-y_{i+1})\mathfrak{e}(\nu) & (\nu_i, \nu_{i+1}) = (\alpha_t, \alpha_{t+1}),\mbox{ for }t\in\mathbb{Z} \\ -(y_i-y_{i+1})\mathfrak{e}(\nu) & (\nu_i, \nu_{i+1}) = (\alpha_t, \alpha_{t-1}),\mbox{ for }t\in\mathbb{Z} \\  0 & \nu_i = \nu_{i+1} \\ \mathfrak{e}(\nu) & \mbox{otherwise} \end{array}\right.\;.
\]

Here, $s_i\cdot \nu\in \mathcal{I}^\beta$ denotes an action of a simple transposition, i.e. a switch of $\nu_i$ with $\nu_{i+1}$.

The algebra $R(\beta)$ becomes ($\mathbb{Z}$-)graded, when setting the degrees
\[
\deg(\mathfrak{e}(\nu)) = 0,\quad \deg(y_i)  = 2,\quad \deg(\psi_i\mathfrak{e}(\nu)) = -(\nu_i, \nu_{i+1})
\]
on the generators.

We write $R(\beta)-\nmod$ ($R(\beta)-\gmod$) for the abelian category of (graded) finite dimensional left modules over $R(\beta)$.

Let $\irr(\beta)$ ($\girr(\beta)$) be the set of isomorphism classes of simple modules in $R(\beta)-\nmod$ ($R(\beta)-\gmod$).

For a graded module $M = (M_i)_{i\in \mathbb{Z}}\in R(\beta)-\gmod$ and an integer $k$, we write
\[
M\langle k \rangle  = (M_{i-k})_{i\in \mathbb{Z}}\in R(\beta)-\gmod
\]
to be the shifted module.

We write $M \mapsto M^{\fgt}$ for the grading-forgetful functor $R(\beta)-\gmod \;\to\; R(\beta)-\nmod$.

Let us note that the forgetful functor gives an evident identification
\begin{equation}\label{eq:homsp}
\oplus_{k\in\mathbb{Z}} \Hom_{R(\beta)-\gmod} (M, N\langle k\rangle) = \Hom_{R(\beta)-\nmod} (M^{\fgt}, N^{\fgt})\;,
\end{equation}
for any $M,N\in R(\beta)-\gmod$.

The algebra $R(\beta)$ possesses an anti-involution $\tau$ given as an identity on all generators in the definition of the algebra. For $M\in  R(\beta)-\gmod$, the complex dual space $M^\ast$ becomes a graded left $R(\beta)$-module through $\tau$.

For a simple $M\in \girr(\beta)$, it is known that there is a (unique) integer $k$, such that $(M\langle k \rangle)^\ast \cong  M\langle k\rangle$. Furthermore, each $M\in \irr(\beta)$ has a unique, up to shift, graded structure.

Thus, we will often treat $\irr(\beta)$ as a subset of $\girr(\beta)$, that is, the isomorphism classes of self-dual simple modules in $R(\beta)-\gmod$.

Given $M\in R(\beta)-\gmod$, we write $[M]\in \mathbb{Z}_{\geq 0} [\girr(\beta)]$ as a formal sum of the Jordan-H\"{o}lder series of $M$. Taking shifts into account, we may write it as a sum
\[
[M]= \sum_{L\in \irr(\beta)} \sum_{i\in \mathbb{Z}} m_{L,i} [L\langle i \rangle]\;.
\]
Given any $L\in \irr(\beta)$ (viewed as a self-dual graded module) and $M\in R(\beta)-\gmod$, we define the \textit{graded multiplicity} of $L$ in $M$ as the Laurent polynomial
\[
m(M,L)(q) = \sum_{i\in \mathbb{Z}} m_{L,i} q^i \in \mathbb{Z}_{\geq 0 }[q,q^{-1}]\;.
\]

\subsection{Restriction and induction}\label{sec:resind}

Following the formalism of \cite{MR2822211}, given $\underline{\beta} = (\beta_1,\ldots, \beta_k)\in (Q_+)^k$, we set the graded algebra
\[
R(\underline{\beta}) = R(\beta_1)\otimes \cdots \otimes R(\beta_k)\;.
\]
The sets $\irr(\underline{\beta})\subset \girr(\underline{\beta})$ and graded multiplicities in $R(\underline{\beta})-\gmod$ are defined analogously. Note, that $\prod_{i=1}^k \irr(\beta_i) \cong \irr(\underline{\beta})$ is a natural bijection given by taking the outer tensor product of simple modules.

Setting $i(\underline{\beta}) = \beta_1 + \ldots +\beta_k$, we have a natural embedding of algebras $\iota_{\underline{\beta}}: R(\underline{\beta}) \to R(\iota(\underline{\beta}))$, as in \cite[Section 2.2]{MR2822211}.

For $\nu_i\in \mathcal{I}^{\beta_i}$, $i=1,\ldots,k$, we have the natural concatenation operation $\nu_1\ast\cdots\ast \nu_k\in \mathcal{I}^{i(\underline{\beta})}$. We then obtain an idempotent element
\[
\mathfrak{e}(\underline{\beta}):= \sum_{\nu_i\in \mathcal{I}^{\beta_i}\, i=1,\ldots,k} \mathfrak{e}(\nu_1\ast\cdots\ast \nu_k) \in R(   i(\underline{\beta}))\;.
\]
Evidently, $\iota_{\underline{\beta}}(1) = \mathfrak{e}(\underline{\beta})$ holds, for the identity element $1\in R(\underline{\beta})$.

Thus, the embedding of algebras gives rise to an exact restriction functor
\[
\rres_{\underline{\beta}}: R(i(\underline{\beta}))-\gmod \;\to\; R(\underline{\beta})-\gmod\;,\quad \rres_{\underline{\beta}}(M) = \mathfrak{e}(\underline{\beta})M\;.
\]

Given $L\in \irr(\underline{\beta})$ and $M\in R(i(\underline{\beta}))-\gmod$, we shortcut notation to
\[
m(M,L)(q):= m(\rres_{\underline{\beta}}(M), L)(q)\;,
\]
for the graded multiplicity in $R(i(\underline{\beta}))-\gmod$.

The restriction functor commits both a left-adjoint induction functor
\[
\iind_{\underline{\beta}}: R(\underline{\beta})-\gmod \;\to\; R(i(\underline{\beta}))-\gmod\;,\quad
\iind_{\underline{\beta}}(M) = R(i(\underline{\beta}))\otimes_{\iota_{\underline{\beta}}(R(\underline{\beta}))}M\;,
\]
and a right-adjoint co-induction functor
\[
\coiind_{\underline{\beta}}: R(\underline{\beta})-\gmod \;\to\; R(i(\underline{\beta}))-\gmod\;,\quad
\coiind_{\underline{\beta}}(M) = \Hom_{R(\underline{\beta})}(R(i(\underline{\beta})),M)\;.
\]
(See, for example, \cite{MR2822211}.)

More generally, when $\underline{\delta} = (\delta^1_1,\ldots, \delta^1_{m_1},\ldots, \delta^k_1,\ldots,\delta^k_{m_k})\in (Q_+)^{\sum_{i=1}^k m_i}$ is such that $\sum_{j=1}^{m_i} \delta^i_j = \beta_i$, for all $1\leq i \leq k$, we have
\[
R(\underline{\delta}) \cong R(\delta^1_1,\ldots, \delta^1_{m_1}) \otimes \cdots \otimes R(\delta^k_1,\ldots,\delta^k_{m_k}) \hookrightarrow R(\underline{\beta})\;,
\]
and the functors
\[
\rres_{\underline{\delta}}^{\underline{\beta}}: R(\underline{\beta})-\gmod\;\to \; R(\underline{\delta})-\gmod\;,\quad \iind_{\underline{\delta}}^{\underline{\beta}},\; \coiind_{\underline{\delta}}^{\underline{\beta}}: R(\underline{\delta})-\gmod\;\to \; R(\underline{\beta})-\gmod
\]
are naturally defined.

Given $M_i \in R(\beta_i)-\gmod$, for $i=1,\ldots k$, we write the induction operation as a product
\[
M_1\circ \cdots \circ M_k = \iind_{\underline{\beta}}(M_1 \otimes \cdots \otimes M_k)\;.
\]
This product equips the larger abelian category
\[
\mathcal{D} = \oplus_{\beta\in Q_+} R(\beta)-\gmod \,
\]
with a monoidal structure.

Restriction, induction and co-induction functors may still be defined for categories of ungraded modules. In particular, the category
\[
\widehat{\mathcal{D}} = \oplus_{\beta\in Q_+} R(\beta)-\nmod \,
\]
retains the monoidal structure of $\mathcal{D}$.

In particular, we note that the functor $\fgt: \mathcal{D}\to \widehat{\mathcal{D}}$ is monoidal.

\subsection{Mackey Theory}\label{sec:mackey}

Let $\underline{\beta}\in (Q_+)^k$ and $\underline{\gamma}\in (Q_+)^l$ be given, for which $i(\underline{\beta}) = i(\underline{\gamma})$ holds. Then, the composition functor
\[
\rres_{\underline{\gamma}}\circ \iind_{\underline{\beta}}
\]
is well-defined, and admits the following typical Mackey-theory description.



Let $M(\underline{\beta},\underline{\gamma})$ be the collection of $k\times l$-matrices $\delta = (\delta_{i,j})$ of elements in $Q_+$, such that $\delta_{i,1} + \ldots + \delta_{i,l} = \beta_i$ holds, for all $1\leq i\leq k$, and $\delta_{1,j} + \ldots + \delta_{k,j}= \gamma_j$ holds, for all $1\leq j\leq l$.

For $\delta\in M(\underline{\beta},\underline{\gamma})$, we define the integer
\begin{equation}\label{eq:deg-form}
\deg(\delta) = -\sum_{ 1\leq i < i' \leq k,\; 1\leq j' < j \leq l} \left( \delta_{i,j}, \delta_{i',j'}\right)\;.
\end{equation}

Each $\delta\in M(\underline{\beta},\underline{\gamma})$ gives rise to two tuples
\[
\underline{\delta}^{row}= (\delta_{1,1}, \ldots,\delta_{1,l}, \delta_{2,1},\ldots,\delta_{2,l},\ldots,\delta_{k,1}, \ldots,\delta_{k,l})\;,
\]
\[
\underline{\delta}^{col}= (\delta_{1,1}, \ldots,\delta_{k,1}, \delta_{1,2},\ldots,\delta_{k,2},\ldots,\delta_{1,l}, \ldots,\delta_{k,l})
\]
in $(Q_+)^{kl}$. There is an obvious isomorphism of algebras $t_\delta: R(\underline{\delta}^{row})\to R(\underline{\delta}^{col})$ by permuting the factors. We set
\[
T_\delta: R(\underline{\delta}^{row})-\gmod\to R(\underline{\delta}^{col})-\gmod
\]
to be the functor obtained from pushing a module through $t_\delta$ and shifting its degree by $\deg(\delta)$.

We set a functor
\[
K_{\delta} :  R(\underline{\beta})-\gmod \;\to\; R(\underline{\gamma})-\gmod\;,\quad K_{\delta}(M) = \iind^{\underline{\gamma}}_{\underline{\delta}^{col}} (T_\delta(\rres^{\underline{\beta}}_{\underline{\delta}^{row}}(M)))\;.
\]

\begin{proposition}\label{prop:mackey}[Restatement of \cite[Theorem 2.1]{mcn-bk}]
For every graded $R(\underline{\beta})$-module $M$, the $R(\underline{\gamma})$-module $\rres_{\underline{\gamma}}(\iind_{\underline{\beta}}(M))$ has a filtration of submodules, whose composition factors are given by $\{K_{\delta}(M)\}_{\delta\in M(\underline{\beta},\underline{\gamma})}$.

\end{proposition}


\subsection{Kleshchev-Ram classification and multisegments}\label{sec:kr-class}

Let us now describe the classification of
\[
\irr_{\mathcal{D}} = \sqcup_{\beta\in Q_+} \irr(\beta)
\]
as obtained in \cite{kr2}. When specialized to this case, it may be viewed as a graded version of the classical Zelevinsky classification of \cite{Zel} (See Section \ref{sec:equiv-cat}).

For each pair of integers $i\leq j$, we denote the element $\Delta(i,j) = \alpha_i + \alpha_{i+1} + \ldots +\alpha_j\in Q_+$. We refer to $\seg = \{\Delta(i,j)\}_{i\leq j} \subset Q_+$ as the set of \textit{segments} (positive roots).

For a segment $\Delta = \Delta(i,j)\in \seg$, we write $i = b(\Delta)$ and $j= e(\Delta)$ for its begin and end points.

Let $\leq$ denote the total lexicographical order on $\seg$, so that $\Delta_1\leq \Delta_2$, if $b(\Delta_1)< b(\Delta_2)$ holds, or that both $b(\Delta_1)= b(\Delta_2)$ and $e(\Delta_1)\leq e(\Delta_2)$ hold.

Similarly, let $\leq_r$ denote the right lexicographical order on $\seg$, defined as $\leq$, but with the roles of $b(\Delta)$ and $e(\Delta)$ reversed.

We refer to elements of the free abelian monoid
\[
\Mult = \mathbb{Z}_{\geq0}[\seg]
\]
as \textit{multisegments}.

We will abuse notation by referring to the set of segments $\Seg$, both as a subset of $\Mult$, and of $Q_+$, depending on context. For example, $\Delta(1,1) + \Delta(1,2), \Delta(1,2)\in\Mult$ are distinct multisegments. Yet, $ \Delta(1,1) + \Delta(2,2)= \Delta(1,2)$ holds as an equation in $Q_+$.

There is a natural additive map $\wt: \Mult \to Q_+$, defined by $\wt(\Delta)= \Delta$, for each $\Delta\in \Seg$.

Each $\Delta\in \seg$ is attached with a \textit{segment module}\footnote{\textit{Cuspidal module} in the language of \cite{kr2}. We refrain from using this terminology here, because of the involvement of supercuspidal representations of $p$-adic groups in the discussion of Section \ref{sec:padic}.} $L_\Delta \in \irr(\Delta)$. It may be characterized as the unique self-dual $1$-dimensional $R(\Delta)$-module, for which there exists $\nu=(\nu_1, \ldots, \nu_{|\Delta|})\in \mathcal{I}^{\Delta}$ with $\nu_1 = \alpha_{e(\Delta)}$, so that $\mathfrak{e}(\nu)L_{\Delta}\neq0$.

Each $\m \in\Mult$ can be uniquely written as $\m = \sum_{i=1}^k p_i \Delta_i$, for segments $\Delta_1 <_r \ldots <_r \Delta_k$ in $\seg$.

In these terms, the Kleshchev-Ram classification attaches to $\m\in \Mult$ the \textit{proper standard module}
\[
\\KR(\gotM) := L_{\Delta_1}^{\circ p_1}\circ \cdots \circ L_{\Delta_k}^{\circ p_k}\;\left\langle {p_1 \choose 2} + \ldots + {p_k \choose 2}\right\rangle\in R(\wt(\m))-\gmod \;.
\]
Here $L^{\circ p} = L\circ \cdots \circ L$ denotes the $p$-th induction product of a module with itself.

\begin{theorem}\label{thm:kr}\cite[Theorem 7.2]{kr2}
The head (or co-socle, i.e. maximal semisimple quotient) of $\\KR(\m)$, denoted as $L_{\gotM}$, is simple and self-dual. The resulting map
\[
\Mult \to \irr_{\mathcal{D}}\;,\quad \m\mapsto L_{\m}
\]
is a bijection.
\end{theorem}

For $\m\in \Mult$, we set
\[
\mathfrak{b}(\m) = n_1\alpha_{b(\Delta_1)} + \ldots +  n_k\alpha_{b(\Delta_k)},\;  \mathfrak{e}(\m) = n_1\alpha_{e(\Delta_1)} + \ldots +  n_k\alpha_{e(\Delta_k)}\in Q_+\;,
\]
and write $|\m| = |\mathfrak{b}(\m)| = |\mathfrak{e}(\m)|$, that is, the number of segments used to define $\m$.

For $L= L_\m\in \irr_{\mathcal{D}}$, we also write $\mathfrak{b}(L)=\mathfrak{b}(\m)$, $\mathfrak{e}(L)=\mathfrak{e}(\m)$ and $\wt(L) = \wt(\m)$.


\section{Representations of $p$-adic $GL_n$}\label{sec:padic}

Let $F$ be a fixed $p$-adic field. We are interested in the smooth representation theory of the sequence of locally compact groups $G_n := GL_n(F)$.

For a $p$-adic group $G$ (such as $G_n$), we write $G-\nmod$ for the abelian category of finite-length (typically infinite-dimensional) smooth representations of $G$ over the complex field. We write $\irr(G)$ for the set of isomorphism classes of irreducible representations in $G-\nmod$.

Let us recall the basic inductive mechanisms for construction and study of representations in $G_n-\nmod$. Most of the tools applied in our discussion stem from the classical texts \cite{BZ1,Zel}. See, for example, \cite{LM2} for a modern exposition.

We say that a tuple of positive integers $\underline{n} = (n_1, \ldots, n_r)$ is a composition of $n$ and write $n = i(\underline{n})$, if $n= n_1+\ldots + n_r$. We denote by $M_{\underline{n}}$ the subgroup of $G_{i(n)}$ isomorphic to $G_{n_1} \times \cdots \times G_{n_r}$ consisting of matrices which are diagonal by blocks of size $n_1, \ldots, n_r$ and by $P_{\underline{n}}$ the subgroup of $G_{i(n)}$ generated by $M_{\underline{n}}$ and the upper
unitriangular matrices. A standard parabolic subgroup of $G_n$ is a subgroup of the form $P_{\underline{n}}$ with $i({\underline{n}}) = n$ and its standard Levi factor is $M_{\underline{n}}$.

We write $\mathbf{i}_{\underline{n}}: M_{\underline{n}}-\nmod\to G_{i({\underline{n}})}-\nmod$ for the exact (normalized) parabolic induction functor associated to $P_{\underline{n}}$.

For $\pi_i\in G_{n_i}-\nmod$, $i=1,\ldots,r$, we write
\[
\pi_1\times\cdots\times \pi_r := \mathbf{i}_{(n_1,\ldots,n_r)}(\pi_1\otimes\cdots\otimes \pi_r)\in G_{n_1+\ldots+n_r}-\nmod\;.
\]
The induction functor $\mathbf{i}_{\underline{n}}$ admits a left-adjoint functor
\[
\mathbf{r}_{\underline{n}}: G_{i(\underline{n})}-\nmod \to M_{\underline{n}}-\nmod
\]
known as the Jacquet functor.

An irreducible representation $\pi\in \irr(G_n)$ is called supercuspidal, if $\mathbf{r}_{\underline{n}}(\pi)=0$, for all ${\underline{n}}$ with $n= i({\underline{n}})$ and $M_{\underline{n}}\neq G_n$. We write $\cusp_n\subset \irr(G_n)$ for the set of supercuspidal irreducible representations.

For any $n$, let $\mu^s= |\det|^s_F,\;s\in \mathbb{C}$ denote the family of one-dimensional representations of $G_n$, where $|\cdot|_F$ is the absolute value of $F$. For $\pi\in G_n - \nmod$, we write $\pi\mu^s := \pi\otimes \mu^s\in G_n-\nmod$.

\subsection{Block decomposition}

Given $\rho\in \cusp_m$ and an integer $d\geq1$, let $\irr^{\mathbb{Z}}_{\rho,d}\subset \irr(G_{md})$ be the set of all irreducible subquotients of representations of the form
\[
(\rho\mu^{k_1}) \times\cdots\times (\rho\mu^{k_d})\;,
\]
for any choice of integers $k_1,\ldots,k_d\in \mathbb{Z}$.

For a pair $(\rho,d)$, we define the \textit{simple line block} $\mathcal{C}(\rho,d)$ to be the Serre subcategory of $G_{md}-\nmod$ consisting of those representations $\pi$ whose irreducible subquotients all belong to $\irr^{\mathbb{Z}}_{\rho,d}$.

Following the general theory of Bernstein blocks (\cite{bern-center}), we may decompose the category of finite-length representations as a sum of abelian categories
\[
G_n-\nmod = \oplus_{\Theta\in \mathfrak{B}_n} \mathcal{C}_\Theta\;,
\]

where each summand stands as a product (in the sense of Deligne \cite[Section 5]{deligne-tan}) of certain simple line blocks of smaller groups:
\[
\mathcal{C}_\Theta \cong \mathcal{C}(\rho_1,d_1) \times \cdots \times \mathcal{C}(\rho_r,d_r)\;.
\]
Here, $\rho_i \in \cusp_{m_i}$, $i=1,\ldots,r$ satisfy $\sum_{i=1}^r m_id_i = n$ and the equivalence\footnote{Note, that it is not claimed that any choice of $\rho_1,\ldots,\rho_r$ produces an equivalence, but rather the existence of such for any $\Theta\in \mathfrak{B}_n$.} is given by parabolic induction. In other words, the functor $\mathbf{i}_{(m_1d_1,\ldots,m_rd_r)}$ becomes an equivalence of categories when restricted to $\mathcal{C}(\rho_1,d_1) \times \cdots \times \mathcal{C}(\rho_r,d_r)$, viewed as a subcategory of $M_{(m_1d_1,\ldots,m_rd_r)}-\nmod$. In particular, we see that the set of irreducible representations in $\mathcal{C}_\Theta$ is naturally described as
\[
\irr_\Theta = \irr^{\mathbb{Z}}_{\rho_1,d_1} \times\cdots\times \irr^{\mathbb{Z}}_{\rho_r,d_r}\;.
\]
In the above sense, simple line blocks contain essentially all of the information encoded in the categories $G_n-\nmod$.


Moreover, simple line blocks are well-aligned relative to the parabolic induction and Jacquet functors. Namely, for any $\rho\in \cusp_m$ and $\underline{d}= (md_1,\ldots,md_r)$, those functors restrict to well-defined pair of adjoint functors between the categories
\begin{equation}\label{eq:resind-diag}
\mathcal{C}(\rho,d_1) \times \cdots \times \mathcal{C}(\rho,d_r) \begin{array}{c} \mathbf{i}_{\underline{d}}   \longrightarrow\\ \longleftarrow \mathbf{r}_{\underline{d}} \end{array}\mathcal{C}(\rho,d_1+\ldots+d_r)\;.
\end{equation}

For these reasons, it is useful to consider the sum of categories
\[
\mathcal{C}_\rho^{\mathbb{Z}} =  \bigoplus_{n=0}^\infty \mathcal{C}(\rho,n)\,
\]
(summing blocks of representations of different groups) and its set of irreducible representations
\[
\irr_\rho^{\mathbb{Z}} = \sqcup_{n=0}^\infty  \irr_{\rho,n}^{\mathbb{Z}}\;.
\]

Given $\pi \in \irr_{\rho,n}^{\mathbb{Z}}$, there exist integers $k_1,\ldots,k_n\in \mathbb{Z}$, for which $\pi$ appears as a sub-representation of $\rho\mu^{k_1}\times \cdots \times \rho\mu^{k_n}$. The sequence $k_1,\ldots, k_n$ is uniquely determined by $\pi$, up to a permutation. The resulting $S_n$-orbit (relative to the usual symmetric group action) on $\mathbb{Z}^n$ is called the supercuspidal support\footnote{We adopt a combinatorial point of view on the general notion of the supercuspidal support of an irreducible smooth representation of a $p$-adic group.} $\supp_\rho(\pi)$ of $\pi$.

For sake of compatibility with other notions appearing in our discussion, let us identify $S_n$-orbits on $\mathbb{Z}^n$ with the elements of $Q_+$ of height $n$, i.e. an orbit represented by $(k_1,\ldots,k_n)$ will correspond to $\alpha_{k_1} + \ldots + \alpha_{k_n}\in Q_+$. In this sense, we will write $\supp_\rho(\pi)\in Q_+$, for $\pi\in \irr_\rho^{\mathbb{Z}}$.

We have a further decomposition of categories of finite-length representations according to their supercuspidal support.

Namely, let $\mathcal{C}_\rho^\beta$ be the full subcategory of $\mathcal{C}(\rho,|\beta|)$ consisting of representations all of whose irreducible subquotients $\pi$ satisfy $\supp_\rho(\pi)= \beta$. Then, we can decompose as
\[
\mathcal{C}(\rho,n) =\bigoplus_{\beta\in Q_+, |\beta|=n} \mathcal{C}_\rho^\beta\;,\quad  \sigma = \oplus \sigma_\beta \;.
\]


Summing over representation categories of groups of different ranks, we obtain a canonical decomposition
\[
\mathcal{C}_\rho^{\mathbb{Z}} = \oplus_{\beta\in Q_+} \mathcal{C}_\rho^{\beta}\;.
\]

For $\underline{\beta} = (\beta_1,\ldots,\beta_r)\in (Q_+)^r$, we may view
\[
\mathcal{C}_\rho^{\underline{\beta}}: = \mathcal{C}_\rho^{\beta_1} \times \cdots \times \mathcal{C}_\rho^{\beta_r}
\]
as a Serre subcategory of $M_{\underline{n}}-\nmod$, where $\underline{n} = (m|\beta_1|,\ldots,m|\beta_r|)$ ($\rho\in \cusp_m$).


The parabolic induction functor $\mathbf{i}_{\underline{n}}$ restricts to a functor
\[
\mathbf{i}_{\underline{\beta}}: \mathcal{C}_\rho^{\underline{\beta}} \to \mathcal{C}_\rho^{i(\underline{\beta})}\;.
\]
In simpler notation, that means that for any given $\pi_i\in  \mathcal{C}_\rho^{\beta_i} $, $i=1,\ldots,r$, we have
\[
\pi_1\times \cdots\times \pi_r\in  \mathcal{C}_\rho^{\beta_1+\ldots+\beta_r}\;.
\]


\subsection{Equivalences of categories}\label{sec:equiv-cat}

The link between the representation theories of the quiver Hecke algebras of Section \ref{sec:quiver} and the $p$-adic groups of this section, passes through the notion of affine Hecke algebras (of type $A$). Let us recall the basic properties of these intermediate categories.

Given $n\in \mathbb{Z}_{>0}$ and a parameter $q\in \mathbb{C}$ (which for our needs will be assumed to be non-root of unity), the root datum of $GL_n$ gives rise to the (extended) \textit{affine Hecke algebra} $H(n,q)$: This is the complex algebra generated by $T_1,\ldots, T_{n-1}$ and invertible $y_1,\ldots,y_n$, subject to the relations
\[
\begin{array}{ll}
T_i T_{i+1} T_i = T_{i+1} T_i T_{i+1},\; & \forall 1\leq i\leq n-2\\

(T_i -q)(T_i+1)=0,\;& \forall 1\leq i \leq n-1\\

T_iT_j = T_j T_i,\;  & \forall |j-i|>1\\

y_iy_j = y_jy_i,\;& \forall 1\leq i,j\leq n\\

T_i y_iT_i = qy_{i+1},\; &\forall 1\leq i\leq n-1\\

T_i y_j = y_jT_i,\; &\forall j\neq i, i+1\;.
\end{array}\;
\]

For a composition $\underline{n} = (n_1, \ldots ,n_r)$ of $n$, the algebra $H(\underline{n},q): =  H(n_1,q)\otimes \cdots \otimes H(n_r,q)$ is naturally embedded as
\[
\iota_{\underline{n}}: H(\underline{n},q) \to H(n,q)\;,
\]
by sending the generators $\tilde{T}_i,\tilde{y}_i$ of $H(n_j,q)$ to $T_{n_1+ \ldots + n_{j-1} + i}, y_{n_1+ \ldots + n_{j-1} + i}$ in $H(n,q)$.

We denote by $\mathcal{M}^q_n$ (respectively, $\mathcal{M}^q_{\underline{n}}$) the category of finite-dimensional modules over the algebra $H(n,q)$ (respectively, $H(\underline{n},q)$).

We have a straightforward decomposition (again, in the sense of the Deligne product of abelian cateogries)
\begin{equation}\label{eq:hecke-decomp}
\mathcal{M}^q_{\underline{n}} = \mathcal{M}^q_{n_1}\times \cdots \times \mathcal{M}^q_{n_r}
\end{equation}
of the abelian category, and an exact induction functor
\[
\iind^q_{\underline{n}}: \mathcal{M}^q_{\underline{n}} \;\to\; \mathcal{M}^q_{n}\;,\quad
\iind^q_{\underline{n}}(M) = H(n,q)\otimes_{\iota_{\underline{n}}(H(\underline{n},q))}M\;.
\]

For $M\in \mathcal{M}^q_n$, the subalgebra $P_n =\mathbb{C}[y_1^{ \pm 1},\ldots y_n^{ \pm 1}]=\iota_{(1,\ldots,1)}(H(1,\ldots,1),q)< H(n,q)$ of Laurent polynomials gives rise to a weight space decomposition of the form
\[
M= \oplus_{\chi\in (\mathbb{C}^\times)^n} M_\chi\;.
\]
Here, $M_\chi$, for $\chi = (\chi_1,\ldots,\chi_n)$, denotes the common generalized eigenspace of $P_n$ in $M$, given by the character $y_i \mapsto \chi_i$.

Recall that the center $Z(H(n,q))$ is given by the symmetric Laurent polynomials in $P_n$ (\cite[Proposition 3.11]{lusz-affine}). In particular, the complex central characters of $H(n,q)$ are given by orbits of the action of $S_n$ (symmteric group) on $(\mathbb{C}^{\times})^n$.

Hence, the full sub-category $\mathcal{M}^{q}_{\overline{\chi}}$, for a given $\overline{\chi} \in (\mathbb{C}^\times)^n/S_n$, of modules $M\in \mathcal{M}^q_n$ that decompose as $M = \oplus_{ \chi\in \overline{\chi}} M_\chi$, is a Serre subcategory.

Alternatively, picking a representative $\chi_0\in \overline{\chi}$ and viewing it as a character of $H((1,\ldots,1),q)$, $\mathcal{M}^{q}_{\overline{\chi}}$ may be described as the full subcategory of $\mathcal{M}^q_n$ consisting of representations all of whose subquotients appear as subquotients in $\iind^q_{(1,\ldots,1)}(\chi_0)$.

Similarly, for $\underline{n}= (n_1,\ldots,n_r)$, the central characters of $H(\underline{n},q)$ are parameterized by $S_{n_1}\times\cdots \times  S_{n_r}$-orbits on $(\mathbb{C}^\times)^{i(\underline{n})}\cong (\mathbb{C}^\times)^{n_1}\times\cdots \times(\mathbb{C}^\times)^{n_r}$. For such given central character $\underline{\chi} = (\overline{\chi}_1,\ldots,\overline{\chi}_r)$ of $H(\underline{n},q)$, we write
\[
\mathcal{M}^q_{\underline{\chi}}=  \mathcal{M}^q_{\overline{\chi_1}}\times \cdots \times \mathcal{M}^q_{\overline{\chi_r}}
\]
for the Serre subcategory of $\mathcal{M}^q_{\underline{n}}$, defined relative to the decomposition \eqref{eq:hecke-decomp}.

We may think of the natural map
\[
i: (\mathbb{C}^\times)^{i(n)}/(S_{n_1}\times \cdots \times S_{n_r}) \to (\mathbb{C}^\times)^{i(n)}/S_{i(n)}
\]
between orbit spaces, as a map from the spectrum of $Z(H(\underline{n},q))$ to the spectrum of $Z(H(i(\underline{n}),q))$.

For $\underline{\chi}\in (\mathbb{C}^\times)^{i(\underline{n})}/(S_{n_1}\times \cdots \times S_{n_r})$, restricting $\iind^q_{\underline{n}}$ to a subcategory, gives an exact functor

\[
\iind^q_{\underline{\chi}}: \mathcal{M}^q_{\underline{\chi}} \to \mathcal{M}^q_{i(\underline{\chi})} \;.
\]

For such $\underline{\chi}$ and a module $M\in \mathcal{M}^q_{i(\underline{n})}$, the space
\[
\rres^q_{\underline{\chi}}(M):= \oplus_{\chi\in \underline{\chi}} M_\chi\;,
\]
is invariant under the action of the subalgebra $\iota_{\underline{n}}(H(\underline{n},q))$. Thus, we may view
\[
\rres^q_{\underline{\chi}}: \mathcal{M}^q_{i(\underline{\chi})} \to \mathcal{M}^q_{\underline{\chi}}
\]
as an exact restriction functor, which is right-adjoint to the induction functor $\iind^q_{\underline{\chi}}$.

We denote by $\mathcal{M}^{q,\mathbb{Z}}_n$ the Serre subcategory of $\mathcal{M}^q_n$ consisting of modules in which $y_i$'s all act with eigenvalues that are integer powers of $q$. In other words,

\[
\mathcal{M}^{q,\mathbb{Z}}_n = \oplus_{\overline{\chi}\in(q^{\mathbb{Z}})^n/S_n} \mathcal{M}^{q}_{\overline{\chi}}\;.
\]

\subsubsection{Rouquier's equivalence}

The assignment
\[
(q^{i_1}, \ldots, q^{i_n})\in (q^{\mathbb{Z}})^n \;\mapsto\; \alpha_{i_1} + \ldots + \alpha_{i_n}\in Q
\]
sets a natural bijection between $(q^{\mathbb{Z}})^n/S_n$ and elements of $Q_+$ of height $n$. Thus, we may view elements of the latter set as central characters of $H(n,q)$. More precisely, for a given $\beta\in Q_+$ we may write $\overline{\chi}_\beta$ for the corresponding character of the algebra $Z(H(|\beta|,q))$.

\begin{theorem}\label{thm:rouq}\cite[Theorem 3.11]{MR2908731}
For each $\beta\in Q_+$ and a non-root of unity $q\in \mathbb{C}^\times$, there is an equivalence of abelian categories
\[
\mathcal{R}_\beta: \mathcal{M}^{q}_{\overline{\chi}_\beta}\to R(\beta)-\nmod\;.
\]

For a module $M\in \mathcal{M}^{q}_{\overline{\chi}_\beta}$, $\mathcal{R}_\beta(M)$ is defined by an action of $R(\beta)$ on the same underlying complex vector space of $M$.

Moreover, the equivalence is compatible with weight space decompositions, in the sense that for $\nu= (\nu_1,\ldots, \nu_{|\beta|})\in \mathcal{I}^\beta$, we have $\mathfrak{e}(\nu)\mathcal{R}_\beta(M) = M_{\chi_\nu}$, where $\chi_\nu = (q^{\nu_1}, \ldots, q^{\nu_{|\beta|}})$.

\end{theorem}

A tuple $\underline{\beta} =(\beta_1,\ldots,\beta_r)\in (Q_+)^r$ defines a central character $\chi_{\underline{\beta}}= (\overline{\chi}_{\beta_1},\ldots,\overline{\chi}_{\beta_r})$ of $H((|\beta_1|,\ldots, |\beta_r|),q)$. Extending the equivalences of Theorem \ref{thm:rouq} to tensor products of algebras, we obtain an equivalence of abelian categories
\[
\mathcal{R}_{\underline{\beta}}: \mathcal{M}^{q}_{\chi_{\underline{\beta}}} \to R(\underline{\beta})-\nmod\;.
\]
Note, that $i(\chi_{\underline{\beta}})= \overline{\chi}_{i(\underline{\beta})}$ clearly holds.

Rouquier's equivalences become compatible with restriction functors in the following sense.

\begin{proposition}
For $\underline{\beta} =(\beta_1,\ldots,\beta_r)\in (Q_+)^r$, we have the identity
\[
\rres_{\underline{\beta}}\circ \mathcal{R}_{i(\underline{\beta})} = \mathcal{R}_{\underline{\beta}}\circ \rres^q_{\underline{\beta}}
\]
of functors.

\end{proposition}
\begin{proof}
The explicit description of the equivalence functor in \cite[Theorem 3.11]{MR2908731} is compatible with restriction on generators of algebras on both sides.
\end{proof}

Since adjoint functors are uniquely defined, we obtain a similar result for the induction operation.

\begin{corollary}\label{cor:ruoq-ind}
For $\underline{\beta} =(\beta_1,\ldots,\beta_r)\in (Q_+)^r$, we have an isomorphism
\[
\mathcal{R}_{i(\underline{\beta})}\circ \iind_{\underline{\beta}}  =  \iind^q_{\underline{\beta}} \circ \mathcal{R}_{\underline{\beta}}
\]
of functors. In particular, for given modules $M_1 \in \mathcal{M}^{q}_{\overline{\chi}_{\beta_1}}$ and $M_2 \in \mathcal{M}^{q}_{\overline{\chi}_{\beta_2}}$, we have an isomorphism of modules
\[
\mathcal{R}_{\beta_1+\beta_2}(\iind_{(\beta_1,\beta_2)}(M_1\boxtimes M_2)) \cong \mathcal{R}_{\beta_1}(M_1)\circ \mathcal{R}_{\beta_2}(M_2)
\]
in $ R(\beta_1+\beta_2)-\nmod$.
\end{corollary}

\subsubsection{Bernstein's equivalence}

Let us recall and explicate some known connections between the representation categories of $p$-adic general linear groups and affine Hecke algebras.

To that end, we should temporarily expand our scope to include representations of infinite-length. Let $G_n-\nmodi$ (respectively, $M_{\underline{n}}-\nmodi$) denote the category of all smooth $G_n$-representations (respectively, $M_{\underline{n}}$-representations) and $\mathcal{N}^q_n$ (respectively, $\mathcal{N}^q_{\underline{n}}$) the category of (possibly infinite-dimensional) $H(n,q)$-modules (respectively, $H(\underline{n},q)$-modules).

A functor
\[
\iind^q_{\underline{n}}: \mathcal{N}^q_{\underline{n}} \;\to\; \mathcal{N}^q_{n}
\]
is defined as in the finite-dimensional case.

For a given $\rho\in \cusp_m$ and an integer $d\geq1$, we define the \textit{simple Bernstein block} $\mathcal{B}(\rho,d)$ to be the full subcategory of $G_{md}-\nmodi$ consisting of those representations whose irreducible subquotients all belong to $\irr_{\rho,d}^\mathbb{C}= \cup_{s\in\mathbb{C}}\irr^{\mathbb{Z}}_{\rho\mu^s,d}$.

Clearly, $\mathcal{B}(\rho,d)$ contains $\mathcal{C}(\rho,d)$ as a full subcategory. Parabolic induction may be defined in the general context of smooth representations. In particular, the analogous functors to \eqref{eq:resind-diag} are well-defined in the context simple Bernstein blocks in place of simple line blocks.

The irreducible smooth representations of $M_{(md_1,\ldots,md_r)}$ are naturally identified with $\irr(G_{md_1})\times \cdots \times \irr(G_{md_r})$. With this view in mind, we may similarly define the simple Bernstein block $\mathcal{B}(\rho, (d_1,\ldots,d_r))$ as the full subcategory of $M_{(md_1,\ldots,md_r)}-\nmodi$ consisting of those representations whose irreducible subquotients all belong to $\irr_{\rho,d_1}^\mathbb{C} \times\cdots \times \irr_{\rho,d_r}^{\mathbb{C}}$.

The following is a major outcome of the type-theory approach to representations of $p$-adic groups.
\begin{theorem}\label{thm:bk93}[Bushnell-Kutzko \cite{bk-book}]
For any $(\rho,d)$ as above, there is an explicit equivalence of abelian categories between $\mathcal{B}(\rho,d)$ and $\mathcal{N}^{q_\rho}_d$.

Here, $q_\rho\in \mathbb{Z}_{>1}$ is a certain power of the residue characteristic of the $p$-adic field defining $G_n$.

\end{theorem}

Restricting the Bushnell-Kutzko equivalences to finite-length representations naturally produce equivalences between $\mathcal{C}(\rho,d)$ and $\mathcal{M}^{q,\mathbb{Z}}_d$. We will outline a construction of such an equivalence through a second approach due to Bernstein \cite{bern-lecture} and Heiermann \cite{heier-cat}. This approach will provide an easier access to compatibility properties with induction functors, which is crucial for our needs.

An object $\pi$ in an abelian category $\mathcal{C}$ is called a \textit{generator}, if the resulting functor $\Hom(\pi, \bullet)$ from $\mathcal{C}$ to right modules over $A_\pi: = \Hom(\pi,\pi)$ (an associative algebra) is an equivalence of categories.

\begin{proposition}\label{prop:bern}[Bernstein]
For $\rho\in \cusp_m$ and an integer $d\geq1$, set $\underline{n}= (m,\ldots,m)$ with $i(\underline{n}) = md$. Suppose that $\tau$ is a finitely-generated generator in $\mathcal{B}(\rho,(1,\ldots,1))$, a subcategory of $M_{\underline{n}}-\nmodi$.

Then, $\mathbf{i}_{\underline{n}}(\tau)$ is a finitely-generated generator for $\mathcal{B}(\rho,d)$.
\end{proposition}

For $\rho\in \cusp_m$ and a choice of a generator $\sigma = \sigma_\rho$ in $\mathcal{B}(\rho,1)$, it is clear that $\sigma^{\boxtimes d} = \sigma\boxtimes \cdots \boxtimes \sigma$ will be a generator for the corresponding category $\mathcal{B}(\rho,(1,\ldots,1))$. Thus, by Proposition \ref{prop:bern} there is an exact equivalence functor
\[
\mathfrak{f}_{\sigma, d}: \mathcal{B}(\rho,d) \;\to \;\nmodi-A(\sigma,d)\;,
\]
where $\nmodi-A(\sigma,d)$ stands for the category of right modules over the finitely generated complex associative algebra
\[
A(\sigma,d): = \Hom_{M_{(m,\ldots,m)}}(\mathbf{i}_{(m,\ldots,m)}(\sigma^{\boxtimes d}),\mathbf{i}_{(m,\ldots,m)}(\sigma^{\boxtimes d}))\;.
\]
Suppose further that $\underline{d}= (md_1,\ldots, md_r)$ is a composition with $i(\underline{d}) = md$ and write $\underline{d}_j = (m,\ldots,m)$ with $i(\underline{d}_j) = md_j$, for every $1\leq j\leq r$. Then,
\[
\sigma^{\underline{d}}:=\mathbf{i}_{\underline{d}_1}(\sigma^{\boxtimes d_1}) \boxtimes \cdots \boxtimes \mathbf{i}_{\underline{d}_r} (\sigma^{\boxtimes d_1})
\]
clearly becomes a generator for the category $\mathcal{B}(\rho,(d_1,\ldots,d_r))$. In other words, we see an exact equivalence
\[
\mathfrak{f}_{\sigma, \underline{d}}: \mathcal{B}(\rho,(d_1,\ldots,d_r)) \;\to \;\nmodi-A(\sigma,\underline{d})\;,
\]
where
\[
A(\sigma,\underline{d}) := \Hom (\sigma^{\underline{d}},\sigma^{\underline{d}}) \cong A(\sigma,d_1)\otimes \cdots \otimes A(\sigma,d_r)\;.
\]
Since $\mathbf{i}_{\underline{d}}(\sigma^{\underline{d}}) = \mathbf{i}_{(m,\ldots,m)}(\sigma^{\boxtimes d})$, the functor $\mathbf{i}_{\underline{d}}$ gives an embedding of endomorphism algebras
\[
\mathbf{i}^A_{\underline{d}}=\mathbf{i}_{\underline{d}}: A(\sigma,\underline{d}) \to A(\sigma,d)\;.
\]
In particular, we obtain an induction functor
\[
\iind^{\sigma}_{\underline{d}}:  \nmodi-A(\sigma,\underline{d})\;\to\; \nmodi-A(\sigma,d)\;,\quad
\iind^{\sigma}_{\underline{d}}(M) =M\otimes_{\mathbf{i}^A_{\underline{d}}( A(\sigma,\underline{d}))}A(\sigma,d)\;.
\]

\begin{proposition}\label{prop:roch}(Roche \cite[5.3]{roche})
The functor diagram
\[
\begin{diagram}
\dgARROWLENGTH=5em
\node{ \mathcal{B}(\rho,(d_1,\ldots,d_r)) } \arrow{s,t}{ \mathbf{i}^A_{\underline{d}} }  \arrow{e,t}{\mathfrak{f}_{\sigma, \underline{d}}}\node{  \nmodi-A(\sigma,\underline{d}) }\arrow{s,t}{ \iind^{\sigma}_{\underline{d}} }\\
\node{\mathcal{B}(\rho,d)  } \arrow{e,t}{  \mathfrak{f}_{\sigma, d} } \node{ \nmodi-A(\sigma,d)  }
\end{diagram}\;,
\]

commutes.
\end{proposition}

In the approach outlined thus far, a link between representations of $p$-adic groups and affine Hecke algebras appears through the following result of Heiermann.

\begin{proposition}\label{prop:heier}[Heiermann \cite{heier-cat}]
For a suitable choice of a generator $\sigma=\sigma_\rho$ as above and any integer $d\geq1$ , there are explicit isomorphisms of complex algebras
\[
h_{\sigma,d}: A(\sigma,d) \to H(d,q_\rho)\;,
\]
for a positive integer $q_\rho>1$.

Moreover, the isomorphisms are compatible with the induction embeddings, in the sense that the diagrams
\[
\begin{diagram}
\dgARROWLENGTH=5em
\node{  A(\sigma,\underline{d}) } \arrow{s,t}{\mathbf{i}^A_{\underline{d}}} \arrow[2]{e,t}{ h_{\sigma,d_1}\otimes \cdots \otimes h_{\sigma,d_r} }\node[2]{ H((d_1,\ldots,d_r),q_\rho)  }\arrow{s,t}{\iota_{(d_1,\ldots,d_r)}}\\
\node{  A(\sigma,d)   } \arrow[2]{e,t}{  h_{\sigma,d}   } \node[2]{  H(d,q_\rho)  }
\end{diagram}\;,
\]
commute, for all $\underline{d}= (md_1,\ldots,md_r)$, with $i(\underline{d})=md$.

\end{proposition}

In particular, the isomorphism $h_{\sigma,d}$ from the above proposition induces an exact equivalence\footnote{A transition between right and left modules is achieved by noting that the relations used to define $H(n,q)$ are symmetric, hence, give rise to a canonical anti-automorphism of the affine Hecke algebra. The observation that this anti-automorphism is compatible with embeddings of the form $\iota_{(d_1,\ldots,d_r)}$ makes the identification trivial with respect to the induction functors involved. } of categories $\mathfrak{h}_{\sigma,d}: \nmodi-A(\sigma,d)\to \mathcal{N}^{q_\rho}_d$. Finally, composing it with the Bernstein equivalence reproduces a desired equivalence
\[
\mathcal{H}_{\rho,d}:=\mathfrak{h}_{\sigma,d}\circ\mathfrak{f}_{\sigma,d}:\mathcal{B}(\rho,d)\to \mathcal{N}^{q_\rho}_d\;,
\]
such as the one obtained in Theorem \ref{thm:bk93}, through separate techniques.

Similarly, for $\underline{d} = (md_1,\ldots,md_r)$ with $i(\underline{d})=md$, we get an equivalence
\[
\mathcal{H}_{\rho,\underline{d}}:\mathcal{B}(\rho,(d_1,\ldots,d_r))\to \mathcal{N}^{q_\rho}_{(d_1,\ldots,d_r)}\;,
\]
by composing $h_{\sigma,d_1}\otimes \cdots \otimes h_{\sigma,d_r}$ with $\mathfrak{f}_{\sigma,\underline{d}}$.

The combination of Propositions \ref{prop:roch} and \ref{prop:heier} now implies a full compatibility of those equivalences with induction functors.

\begin{corollary}\label{cor:ind-comp}
The functor diagram
\[
\begin{diagram}
\dgARROWLENGTH=5em
\node{ \mathcal{B}(\rho,(d_1,\ldots,d_r)) } \arrow{s,t}{ \mathbf{i}_{\underline{d}} }  \arrow{e,t}{\mathcal{H}_{\sigma, \underline{d}}}\node{  \mathcal{N}^q_{(d_1,\ldots,d_r)} }\arrow{s,t}{ \iind^q_{(d_1,\ldots,d_r)} }\\
\node{\mathcal{B}(\rho,d)  } \arrow{e,t}{  \mathcal{H}_{\sigma, d} } \node{ \mathcal{N}^{q_\rho}_d  }
\end{diagram}\;,
\]

commutes.
\end{corollary}

\begin{proposition}
For each $\beta\in Q_+$, the functor $\mathcal{H}_{\rho,|\beta|}$ restricts to an equivalence
\[
\mathcal{H}_{\rho,\beta}: \mathcal{C}^\beta_\rho \to \mathcal{M}^{q_\rho}_{\overline{\chi}_{\beta}}
\]
between abelian categories with finite-length objects.

Similarly, for $\underline{\beta} = (\beta_1,\ldots,\beta_r)\in (Q_+)^r$, the functor $\mathcal{H}_{\rho,(|\beta_1|,\ldots,|\beta_r|)}$ restricts to an equivalence
\[
\mathcal{H}_{\rho,\underline{\beta}}: \mathcal{C}^{\underline{\beta}}_\rho \to \mathcal{M}^{q_\rho}_{\chi_{\underline{\beta}}}\;.
\]
\end{proposition}
\begin{proof}
As detailed in \cite[Proposition 3.2]{me-restriction}, for $s\in \mathbb{C}^\times$, the functor $\mathcal{H}_{\rho,1}$ takes the irreducible representation $\rho\mu^s$ to the character $y_1 \mapsto q_\rho^s$ of $H(q_\rho, 1)$. This is the content of the desired statement for the case of $|\beta|=1$.

For general $\beta$, the statement now follows easily from Corollary \ref{cor:ind-comp} and the characterizations of all of the involved subcategories by induction functors.

\end{proof}

Finally, for given $\rho\in \cusp$ and $\beta\in Q_+$, composing the Bernstein and Rouquier equivalences, we obtain a direct link between smooth representations of $p$-adic groups and modules over quiver Hecke algebras.

\begin{theorem}\label{thm:fullequi}
For $\rho\in \cusp$ and $\beta\in Q_+$, there is an exact functor
\[
\mathcal{E}_{\rho,\beta}:= \mathcal{R}_\beta \circ \mathcal{H}_{\rho,\beta}:  \mathcal{C}^\beta_\rho  \to R(\beta)-\nmod\;,
\]
which gives an equivalence of abelian categories.

Summing $\mathcal{E}_\rho := \oplus_{\beta\in Q_+} \mathcal{E}_{\rho,\beta}$ gives an equivalence between $\mathcal{C}^{\mathbb{Z}}_\rho$ and $\widehat{\mathcal{D}}$.

\end{theorem}

Similarly,
\[
\mathcal{E}_{\rho,\underline{\beta}}:= \mathcal{R}_{\underline{\beta}} \circ \mathcal{H}_{\rho,\underline{\beta}}:  \mathcal{C}^{\underline{\beta}}_\rho  \to R(\underline{\beta})-\nmod
\]
becomes an equivalence of abelian categories.

\begin{proposition}\label{prop:monoid}
For $\rho\in \cusp$, the equivalence functor $\mathcal{E}_\rho$ is monoidal.

In particular, for representation $\pi_1\in \mathcal{C}^{\beta_1}_\rho$ and $\pi_2\in \mathcal{C}^{\beta_2}_\rho$, we have an isomorphism
\[
\mathcal{E}_{\rho,\beta_1+\beta_2}(\pi_1\times \pi_2)\cong \mathcal{E}_{\rho,\beta_1}(\pi_1)\circ \mathcal{E}_{\rho,\beta_2}(\pi_2)
\]
of $R(\beta_1+ \beta_2)$-(ungraded)-modules.

\end{proposition}
\begin{proof}
Corollaries \ref{cor:ruoq-ind} and \ref{cor:ind-comp},
\end{proof}

\subsection{Irreducible representations}

For a given $\rho\in \cusp_m$, the equivalence $\mathcal{E}_\rho$ gives rise to a bijection between the sets $\irr_\rho^{\mathbb{Z}}$ and $\irr_{\mathcal{D}}$ (See Section \ref{sec:quiver}). The resulting bijection indentifies the Zelevinsky classification with the Kleshchev-Ram description of \ref{thm:kr}.

Explicitly, for a segment $\Delta= \Delta(a,b)\in \seg$, the induced representation
\[
\rho\mu^a\times\rho\mu^{a+1}\times\cdots\times \rho\mu^{b} \in \mathcal{C}_\rho^{\Delta}
\]
has a unique irreducible quotient, which we write as $Z(\Delta)\in \irr_{\rho,b-a+1}^{\mathbb{Z}}$. Consequently, by Proposition \ref{prop:monoid}, $\mathcal{E}_\rho(Z(\Delta))\in \irr(\Delta)$ becomes an irreducible quotient of
\[
\delta_a\circ \delta_{a+1}\circ \cdots \circ \delta_b\;,
\]
where $\delta_i := \Delta(i,i) = \mathcal{E}_\rho(\rho\mu^i)$ is the unique simple $R(\alpha_i)$-module.

It follows that $\mathcal{E}_\rho(Z(\Delta))\cong L_\Delta^{\fgt}$ is the (ungraded) segment module.

Now, for a multisegment $\m= \sum_{i=1}^k \Delta_i \in\mathfrak{M}$ with $\Delta_1 \leq_r \ldots \leq_r \Delta_k$ in $\seg$, it follows from Proposition \ref{prop:monoid} that in $ R(\wt(\m))-\nmod$,
\[
\mathcal{E}_\rho(Z(\Delta_1)\times \cdots \times Z(\Delta_k)) \cong \\KR(\m)^{\fgt}
\]
holds.

Hence, the unique irreducible quotient $Z(\m)\in \mathcal{C}_\rho^{\wt(\m)}$ of $Z(\Delta_1)\times \cdots \times Z(\Delta_k)$ must satisfy
\[
\mathcal{E}_\rho(Z(\m)) = L_{\m}^{\fgt},
\]
as ungraded isomorphism classes in $\irr(\wt(\m))$.

\section{Special classes of modules}\label{sec:special}

Given segments $\Delta_1,\Delta_2\in \Seg$, we write $\Delta_1 \prec \Delta_2$, if $b(\Delta_1)<b(\Delta_2)$, $e(\Delta_1) <e(\Delta_2)$ and $e(\Delta_1)\geq b(\Delta_2)-1$ hold. We say that the pair of segments $(\Delta_1,\Delta_2)$ is \textit{linked}, if either $\Delta_1\prec \Delta_2$, or $\Delta_2\prec\Delta_1$.

It is known that $L_{\Delta_1}\circ L_{\Delta_2}$ is a simple module, if and only if, the pair $(\Delta_1,\Delta_2)$ is not linked. (For example, by applying the $\mathcal{E}_\rho$ functor and deducing the fact from standard Zelevinsky theory).

\subsection{Indicator modules}\label{sec:indic}

Let us formalize a point of view, which is recurs often in various treatments in literature, such as \cite{me-decomp}, \cite{MR2811321}, \cite{MR1923974}.

For a choice of integers $b_1\geq \ldots \geq b_k \geq a$, we set a multisegment
\[
\gotM(a\,;b_1,\ldots,b_k) = \Delta(a,b_1)+\ldots + \Delta(a,b_k)\in \Mult \;.
\]
We call such multisegments \textit{left-aligned}.

For a left-aligned multisegment $\gotM= \gotM(a\,;b_1,\ldots,b_k)$, we define the simple module
\[
\nabla(\gotM) = \nabla(a\,;b_1,\ldots,b_k)=  L_{\Delta(a,b_1)} \circ\cdots \circ L_{\Delta(a,b_k)}\left\langle {k \choose 2}\right\rangle\in \girr(\wt(\m))\;,
\]
and write $b(\gotM(a\,;b_1,\ldots,b_k)) = a$.

Every $0\neq \gotM\in\Mult$ clearly admits a unique decomposition as $\gotM = \gotM_1+ \ldots + \gotM_l$, where $\gotM_i,\;i=1,\ldots,l$ are left-aligned multisegments, such that $b(\gotM_1) < \ldots < b(\gotM_l)$.

Following \cite{me-decomp}, we define the \textit{indicator module}
\[
L_{\gotM}^{\otimes} = \nabla(\gotM_1)\boxtimes \cdots \boxtimes \nabla(\gotM_l)\in \girr(\underline{\beta})\;,
\]
where $\underline{\beta} = (\wt(\m_1), \ldots, \wt(\m_l))$, and set
\[
\Sigma(\m) = \iind_{\underline{\beta}}(L_{\gotM}^{\otimes})\in R(\wt(\m))-\gmod\;.
\]
For $M = L_\m \in \irr_{\mathcal{D}}$, we also write $M^\otimes = L_{\gotM}^{\otimes}$.

\begin{lemma}\label{lem:unlink-seg}
Suppose that $\Delta_1 <_r \Delta_2$ are unlinked segments, that is, $\Delta_1\nprec \Delta_2$ and $\Delta_2\nprec \Delta_1$. We have
\[
L_{\Delta_2}\circ L_{\Delta_1} \cong \left\{\begin{array}{cc} L_{\Delta_1}\circ L_{\Delta_2}\;\langle -1 \rangle & b(\Delta_1) =  b(\Delta_2)\mbox{ or }e(\Delta_1) = e(\Delta_2) \\
 L_{\Delta_1}\circ L_{\Delta_2}& \mbox{otherwise} \end{array} \right.\;,
\]
in $R(\wt(\Delta_1+\Delta_2))-\gmod$.
\end{lemma}
\begin{proof}
The module $L_{\Delta_1}\circ L_{\Delta_2}= \KR(\Delta_1 + \Delta_2)$ is simple. It follows from Theorem \ref{thm:kr} that $L_{\Delta_1}\circ L_{\Delta_2}= L_{\Delta_1 + \Delta_2}$ is self-dual. From \cite[Theorem 2.2]{MR2822211}, we obtain
\[
L_{\Delta_1}\circ L_{\Delta_2}\cong L^\ast_{\Delta_2}\circ L^\ast_{\Delta_1} \langle (\Delta_1,\Delta_2)\rangle =L_{\Delta_2}\circ L_{\Delta_1} \langle (\Delta_1,\Delta_2)\rangle  \;,
\]
and the statement follows from a simple computation of the bilinear form on $Q$.
\end{proof}

\begin{lemma}\label{lem:kr-indi}
For every $\gotM\in\Mult$, the proper standard module $\KR(\gotM)$ and $\Sigma(\gotM)$ are isomorphic in $R(\wt(\m))-\gmod$.
\end{lemma}

\begin{proof}
Both modules may be presented as a convolution product of the same segment modules defined by $\m$.

It is a simple corollary of Lemma \ref{lem:unlink-seg} that the isomorphic class of $S(\gotM)$ is not affected by changing the order $<_r$ into $<$ in its definition.

Reversing the product order, Lemma \ref{lem:unlink-seg} also implies that
\[
\nabla(a\,;b_1,\ldots,b_k) \cong L_{\Delta(a,b_k)} \circ\cdots \circ L_{\Delta(a,b_1)}\left\langle {r \choose 2}\right\rangle\;,
\]
where $r= \#\{(i\neq j\;:\; b_i=b_j\}$, for any $b_1\geq\ldots\geq b_k\geq a$.

The isomorphism is now evident when comparing both constructions.

\end{proof}

\begin{corollary}\label{cor:indic-self}
For every simple graded self-dual module $M\in \irr_{\mathcal{D}}$, the indicator module $M^{\otimes}$ appears as a submodule of $\rres_{\underline{\beta}}(M)$, for an appropriate $\underline{\beta}$.
\end{corollary}
\begin{proof}
By Lemma \ref{lem:kr-indi}, $M$ appears as a quotient of $\iind_{\underline{\beta}}(M^\otimes)$.  The statement follows from an adjunction of functors.
\end{proof}

\subsection{Spherical modules}

We say that a graded self-dual simple module $L = L_\m \in \irr_{\mathcal{D}}$, for $\m=\Delta_1+ \ldots + \Delta_k\in\Mult$, is a \textit{spherical} module, if all pairs of segments $(\Delta_i,\Delta_j), i,j=1,\ldots,k$, are not linked.

By exactness of induction functors, it follows that $L_\m\in \irr$ is spherical, if and only if, its associated proper standard module $\KR(\m)$ is simple, that is, $\KR(\m)\cong L_\m$.

It is evident that for each $\beta\in Q_+$, there is a unique spherical self-dual
\[
L^{Sph}(\beta)= L_{\m^{Sph}(\beta)}\in \irr(\beta)\;.
\]
In other words, $\m^{Sph}(\beta)\in \Mult$ is the unique spherical multisegment with $\wt(\m^{Sph}(\beta))= \beta$.


\begin{proposition}\label{prop:sph}
For $\underline{\beta} = (\beta_1,\ldots,\beta_k) \in Q_+^k$, and a simple module $L = L_1\boxtimes \cdots \boxtimes L_k\in \girr(\underline{\beta})$, the graded multiplicity
\[
m(\iind_{\underline{\beta}}(L), L^{Sph}(i(\underline{\beta})))(q)
\]
is either a monomial $q^r$, when $L_i^{\fgt} \cong L^{Sph}(\beta_i)$ for all $i=1,\ldots,k$, or the zero polynomial, otherwise.

\end{proposition}
\begin{proof}
A representation of a $p$-adic group $\pi\in \irr(G_n)$ is said to be spherical if it has a non-zero vector invariant under the action of a maximal compact subgroup $K_n < G_n$.

For $\rho\in \cusp_m$, by the Zelevinsky classification, an irreducible representation $Z(\m)\in \mathcal{C}_\rho^{\beta}$ is spherical, if and only if, $\m = \m^{Sph}(\beta)$ (\cite{MR850742}). In other words, our definition of spherical modules in $\irr_{\mathcal{D}}$ is made so that the equivalence $\mathcal{E}_\rho$ respects sphericity\footnote{An alternative argumentation would be to apply the well-known Zelevinsky involution in the $p$-adic setting, which exchanges the notion of spherical irreducible representations with that of Whittaker-generic representations. Same properties for the latter notion are classical results of \cite{rodier}.}.

For any $\pi_1, \ldots, \pi_k\in \irr_\rho^{\mathbb{Z}}$, it is known that the parabolic induction representation $\pi_1\times\cdots\times \pi_k\in \mathcal{C}_{\rho}^\beta$ contains a spherical irreducible subquotient, if and only if, all $\pi_1,\ldots,\pi_k$ are spherical. Moreover, a spherical irreducible quotient may appear at most once in the Jordan-H\"{o}lder series of the induced representation (See, for example, \cite[Lemma 4.1]{MR3194013}).

Thus, for any choice of $\rho\in \cusp_m$, applying the equivalence $\mathcal{E}_\rho$ and using its monoidality as in Proposition \ref{prop:monoid}, shows that (the ungraded multiplicity) commits
\[
m(\iind_{\underline{\beta}}(L), L^{Sph}(i(\underline{\beta})))(1)\in \{0,1\}\;,
\]
and that it is non-zero only when $L^{\fgt}_1,\ldots,L^{\fgt}_k$ are all spherical.

\end{proof}

\subsection{Homogeneous modules}\label{sec:homog}

A module $M= (M_i)_{i\in \mathbb{Z}}\in R(\alpha)-\gmod$ is called \textit{homogeneous}, if it is concentrated at one degree, that is, $M_i= \{0\}$, for all $i\neq i_M$.

In \cite{kr1}, all homogenous modules in $\irr_{\mathcal{D}}$ were classified. As it turned out, passing through the equivalence $\mathcal{E}_{\rho}$ identifies the notion of irreducible homogeneous modules with that of \textit{ladder} representations in $\mathcal{C}^{\mathbb{Z}}_\rho$.

Let us recall the classification in the terms that are more familiar in the $p$-adic groups literature.

For two sequences of integers $\lambda = ( \lambda_1 > \ldots > \lambda_r)$ and $\mu = (\mu_1 > \ldots > \mu_r)$, which satisfy $\lambda_i\leq \mu_i$, for $1\leq i\leq r$, we set the (ladder) multisegment
\[
\m(\lambda,\mu) = \Delta(\lambda_1,\mu_1-1) + \ldots + \Delta(\lambda_r,\mu_r-1) \in \mathcal{M},
\]
and define
\[
\Xi(\lambda, \mu): = L_{\m (\lambda,\mu)} \in \irr(\wt(\m(\lambda,\mu)))\;.
\]
Representations of the form $Z(\m(\lambda,\mu))$ are the ladder representations for $p$-adic groups. It follows from the study in \cite{LapidKret}, that the Kleshchev-Ram construction coincides with that of ladder representations. Thus, by \cite[Theorem 3.6]{kr1}, the collection $\{\Xi(\lambda, \mu)\}_{\lambda,\mu}$ exhausts all self-dual irreducible homogeneous modules in $\mathcal{D}$.

One particularly convenient property of the homogeneous class is that their restrictions are homogenous as well, and are easily described. The following may be viewed as a combined statement of \cite[Section 3.4]{kr1} and \cite{LapidKret}.
\begin{proposition}\label{prop:ladderjac}
Given $\underline{\beta} \in (Q_+)^k$, and a homogeneous module $\Xi(\lambda, \mu)\in \irr(i(\underline{\beta}))$ with $\lambda = ( \lambda_1 > \ldots > \lambda_r)$ and $\mu = (\mu_1 > \ldots > \mu_r)$, we have
\[
\rres_{\underline{\beta}}(\Xi(\lambda, \mu) )= \bigoplus_{ \nu^1, \ldots, \nu^{k-1}} \Xi (\nu^{k-1}, \mu)\boxtimes \Xi (\nu^{k-2}, \nu^{k-1})\boxtimes \cdots \boxtimes \Xi (\lambda, \nu^1)\;,
\]
as a graded $R(\underline{\beta})$-module.

Here the sum is taken over all possible sequences $\nu^i = (\nu^i_1> \ldots >\nu^i_r)$ of integers, such that $\lambda_j \leq \nu_j^1\leq \ldots \leq \nu_j^{k-1} \leq \mu_j$ holds for every $1\leq j \leq r$, and that the resulting representation is in $R(\underline{\beta})-\gmod$.

\end{proposition}

A special case of a homogeneous module is a segment module: $\Xi((a),(b)) = L_{\Delta(a,b-1)}$. In this case, we see from Proposition \ref{prop:ladderjac} that $\rres_{\underline{\beta}}(L_{\Delta(a,b)})$ is either $0$ or an irreducible module given by
\begin{equation}\label{eq:seg-jac}
L_{\Delta(c_k , b-1)} \boxtimes L_{\Delta(c_{k-1} , c_{k}-1)}  \boxtimes \cdots \boxtimes L_{\Delta(a,c_1-1)} \;,
\end{equation}
for integers $a\leq c_1 \leq \ldots\leq c_k\leq b$.

\section{RSK for multisegments}\label{sec:rsk}

We recall the combinatorial algorithms associated with the Robinson-Schensted-Knuth correspondence, in a form adapted to our setting.

Given segments $\Delta_1, \Delta_2\in\Seg$, we write $\Delta_1\smlr\Delta_2$, if $b(\Delta_1)<b(\Delta_2)$
and $e(\Delta_1)<e(\Delta_2)$ hold. This is a strict partial order on $\Seg$.

In these terms, we say that a multisegment
\[
0\neq \la=\sum_{i=0}^k\Delta_i \in \Mult
\]
is a \textit{ladder multisegment}, if $\Delta_i\smlr\Delta_{i-1}$, for $i=1,\ldots,k$.

We write $\Lad\subset \Mult$ for the collection of all ladder multisegments.

For any $0\neq \m \in \Mult$, we set its \textit{width} $\omega(\m)$ to be the minimal number of ladder multisegments $\la_1, \ldots, \la_{\omega(\m)}\in\Lad$, for which we can decompose as $\m = \la_1 + \ldots + \la_{\omega(\m)}$.

We write
\[
\Mult \setminus \{0\} = \bigcup_{i=1}^\infty \Mult_i\;,\quad  \Mult_i = \{0\neq \m \in\Mult\;:\; \omega(\m)= i\}  \;.
\]
Note, that $\Lad = \Mult_1$.

Suppose that a multisegment
\[
0\neq \m = \sum_{i\in I}\Delta_i\in \Mult
\]
and a ladder multisegment
\[
\la =  \sum_{j\in J} \Delta'_{j} \in \la
\]
are given. We may write $J = \{j_1,\ldots j_l\}$, with $\Delta_{j_{l}} \smlr\ldots\smlr \Delta_{j_1}$.

We say that the pair $(\la,\m)$ is \textit{permissible}\footnote{ A slightly different formulation was used in \cite{gur-lap}. The equivalence of conditions is a straightforward exercise.}, if for every choice of indices $i_1,\ldots, i_m\in I$, for which $\Delta_{i_m} \smlr \ldots \smlr \Delta_{i_1}$ holds (sub-ladder of $\m$), there is an injective increasing function
\[
\phi: \{1,\ldots,m\} \to \{1,\ldots,l\}\;,
\]
for which $\lshft\Delta_{i_t} \prec \Delta_{j_{\phi(t)}}$ holds, for all $1\leq t\leq m$.

Here we denote $\lshft\Delta=\Delta(a-1,b-1)\in \Seg$, for a segment $\Delta = \Delta(a,b)\in \Seg$.


Let $\pairs \subset \Lad \times \Mult$ be the collection of permissible pairs.

In further refinement, we write $\pairs = \bigcup_{i=1}^\infty \pairs_i$, where $\pairs_i \subset \Lad\times \Mult_i$ are the permissible pairs $(\la,\m)$, with $\omega(\m)=i$.

\begin{proposition}\cite[Proposition 2.4]{gur-lap}
There is a bijection
\[
\Vien:\Mult\setminus\{0\} \rightarrow\pairs\;,
\]
which is explicitly given by the Knuth-Viennot implementation of the RSK correspondence.
\end{proposition}

The combinatorial algorithm defining $\Vien$, which was described in detail in \cite[Section 2.2.2]{gur-lap}, is a manifestations of the Knuth algorithm from \cite{MR0272654}.

\begin{proposition}
The restriction of the map $\Vien$ to the subset $\Mult_i$, for $i\geq2$, results in a bijection
\[
\Vien_i : \Mult_i \rightarrow \pairs_{i-1}\subset \Lad \times \Mult_{i-1}\;.
\]
\end{proposition}
\begin{proof}
As was discussed in \cite[Remark 4.4]{gur-lap}, the width $\omega(\m)$ of a multisegment $\m\in \Mult$ may be given a combinatorial interpretation using the results of \cite{me-decomp}. Standard properties of the Knuth map then imply that $\omega(\m') = \omega(\m)-1$, whenever $\Vien(\m) = (\la,\m')$.
\end{proof}

Given a multisegment $\m\in \Mult_d$, we may apply the map Knuth-Viennot map recursively:
\[
\Vien(\m)= (\la_1,\m_1),\; \Vien(\m_1) = (\la_2, \m_2),\;\ldots , \Vien(\m_{d-2}) = (\la_{d-1}, \la_d) \in \pairs_1 \subset \Lad\times \Lad\;.
\]
We take the resulting $d$ ladders
\[
\RSK(\m) = (\la_1, \la_2,\ldots,\la_d)\in \Lad^d
\]
as the \textit{RSK-transform} of $\m$.

Let us reformulate the information encoded in $\RSK(\m)$ into a combinatorial presentation. Recall (Section \ref{sec:homog}) that ladder multisegments segments are uniquely described as
\[
\la_i = \m(c_i,d_i)\;,
\]
for $i=1,\ldots, d$, and given tuples of integers
\[
c_i = (c_{i,1} > \ldots > c_{i,\lambda_i}),\quad d_i = (d_{i,1} > \ldots >d_{i,\lambda_i})\;.
\]

We know (again, by \cite[Proposition 2.4]{gur-lap}) that $(\la_i,\la_j)\in \pairs$, for all $1\leq i<j\leq d$. It then easily follows from the permissibility condition that $\lambda_1 \geq \ldots \geq \lambda_d$, and that $c_{1,j} \geq c_{2,j} \geq \ldots$ holds, for each index $j$.

In particular, we obtain a pair of (inverted\footnote{Strictly descending rows, and weakly descending columns.}) semi-standard Young tableaux\footnote{Note, that our current convention is slightly different from that of \cite{gur-lap}: Our $d_{i,j}$ stands for $d_{i,j}+1$ in the conventions of that source.}

\ytableausetup{mathmode,boxsize=2em}
\[
P_{\m}=\begin{ytableau}
c_{1,1} & c_{1,2} & \dots & c_{1,\lambda_2}
& \dots & c_{1,\lambda_1} \\
c_{2,1}    & c_{2,2} & \dots
& c_{2,\lambda_2} \\
\vdots & \vdots
& \vdots \\
c_{k,1} & \dots & c_{k,\lambda_k}
\end{ytableau}\quad ,\quad
\ \ \
Q_{\m}=\begin{ytableau}
d_{1,1} & d_{1,2} & \dots & d_{1,\lambda_2}
& \dots & d_{1,\lambda_1} \\
d_{2,1}    & d_{2,2} & \dots
& d_{2,\lambda_2} \\
\vdots & \vdots
& \vdots \\
d_{k,1} & \dots & d_{k,\lambda_k}
\end{ytableau}
\]

of equal shape, whose rows are given by the partition $\lambda(\m) = (\lambda_1,\ldots,\lambda_d)$ of the integer $|\la_1|+\ldots + |\la_d| $.

\begin{proposition}\label{eq:preser}
  For any $0\neq \m\in \Mult$ with $\RSK(\m) = (\la_1,\ldots,\la_d)$, we have the equalities
 \[
\mathfrak{b}(\m) = \mathfrak{b}(\la_1) +\ldots+ \mathfrak{b}(\la_d),\quad \mathfrak{e}(\m) = \mathfrak{e}(\la_1) + \ldots+  \mathfrak{e}(\la_d)
\]
in $Q_+$.

In particular, $\lambda(\m)$ is a partition of the integer $|\m|$.
\end{proposition}
\begin{proof}
It follows directly from the Knuth algorithm description (or from the description of $\Vein$ explicated in the next section) that for $\Vien(\m) = (\la_1,\m')$, we have
 \[
\mathfrak{b}(\m) = \mathfrak{b}(\la_1) + \mathfrak{b}(\m'),\quad \mathfrak{e}(\m) = \mathfrak{e}(\la_1) +  \mathfrak{e}(\m')\;.
\]
Since $\RSK(\m') = (\la_2,\ldots,\la_d)$, the statement follows inductively.
\end{proof}

\subsection{Algorithm description}\label{sec:algo}

In this work, we will be interested in a (equally explicit) description of the inverse map $\Vein : \pairs \to \Mult\setminus \{0\}$ to $\Vien$. Let us describe the map in detail.

Let $(\la,\m)\in \pairs$ be given.

Again, we write $\m = \sum_{i\in I}\Delta_i$ and $\la =  \sum_{j\in J} \Delta_{j} \in \Lad$ for disjoint index sets $I,J$.

Let us denote $J = \{j_1,\ldots j_l\}$, with $\Delta_{j_{l}} \smlr\ldots\smlr \Delta_{j_1}$.

We also assume that $I$ is linearly ordered by a fixed relation $<$, satisfying
\begin{equation}\label{eq:ordered}
b(\Delta_{i_1}) > b(\Delta_{i_2}) \quad \text{ or } \quad \left\{\begin{array}{l} b(\Delta_{i_1}) = b(\Delta_{i_2}) \\ e(\Delta_{i_1}) \leq  e(\Delta_{i_2})\end{array}\right.\;,
\end{equation}
for all $i_1> i_2$ in $I$.

For each $i\in I$, we take note of the number
\[
\depth'_{\la,\m}(i) = \max\{t :\lshft\Delta_i\smlr\Delta_{j_t}\}\;.
\]

Next, for each $i\in I$, we define
\[
\depth_{\la,\m}(i) =\min\{\depth'_{\la,\m}(i_k)-k:\exists i=i_0,\dots,i_k\in I \text{ such that }\Delta_{i_{r+1}}\smlr\Delta_{i_r}, r=0,\dots,k-1\}\;.
\]
For convenience, we extend the domain of the function $\depth_{\la,\m}$ to $I\cup J$, by setting $\depth_{\la,\m}(j_t) = t$.

Let $\sigma$ be a permutation on the index set $I\cup J$, given by its decomposition into the following disjoint cycles:
\[
(i_1, \ldots , i_s, j_t),\;
\]
for each $1\leq t\leq  l$, where
\[
\{i_1 > \ldots > i_s\}\cup \{j_t\} = \depth_{\la,\m}^{-1}(t)\;.
\]
For $i\in I\cup J$, we write $i_\# = \sigma(i)$.

We can now set a new multisegment
\[
\Vein(\la,\m)=\sum_{i\in I\cup J}\Delta^{\clubsuit}_i\in \Mult,\;
\]
by defining $\Delta^{\clubsuit}_i=\Delta(b(\Delta_{i}),e(\Delta_{i_\#}))\in \Seg$.\footnote{Note the slight difference in notation from \cite{gur-lap}.}

\subsubsection{Properties of $\Vein$.}

Note, that $e(\Delta_i)\leq  e( \Delta_{i_\#})= e(\Delta^{\clubsuit}_i)$, for all $i\in I$, while $e(\Delta_{j_\#}) \leq e(\Delta_{j})$, for $j\in J$.

For $i\in I\cup J$, we set $i^\vee = \sigma^{-r}(i)$, where $r\geq 1$ is the minimal power for which
\[
e(\Delta^{\clubsuit}_{\sigma^{-r}(i)}) \neq e(\Delta_{\sigma^{-r}(i)})
\]
holds, or $i^\vee = j_{\depth_{\la,\m}(i)}$, if such power does not exist.

Note, that $e(\Delta^{\clubsuit}_{i^\vee}) = e(\Delta_i)$.

When $i,i^\vee\in I$, we have $i\leq  (i^\vee)_\# < i^\vee$ and $e(\Delta_i) = e(\Delta_{(i^\vee)_\#})$.

When $i,i_\#, (i_\#)^\vee \in I$, we have $i_\# < i \leq (i_\#)^\vee$ and either $e(\Delta_{(i_\#)^\vee})= e(\Delta_{i})< e(\Delta_{i_\#})$ or $e(\Delta_{(i_\#)^\vee}) < e(\Delta_{i})= e(\Delta_{i_\#})$.

\begin{lemma}\label{lem:techdep}
Let $i\in I$ be an index and $n$ an integer, with $\depth_{\la,\m}(i) < n \leq \depth'_{\la \m}(i)$. Then, there exists $i(n)\in I$, such that $\Delta_{i(n)} \smlr \Delta_i$ and $\depth_{\la,\m}(i(n))=n$.

\end{lemma}

\begin{proof}

Let us write $m = \depth'_{\la, \m}(i) - \depth_{\la,\m}(i)$, and $p = n- \depth_{\la ,\m}(i) \leq m$.

Suppose that $\Delta_{i_k} \smlr \ldots \smlr \Delta_{i_1} \smlr \Delta_{i_0}$ are segments, such that $i_0=i$ and $\depth'_{\la,\m}(i_k) = \depth_{\la,\m}(i) + k$.  Since $\depth'_{\la,\m}(i) \leq \depth'_{\la,\m}(i_k) $, we must have $m\leq k$.

Now, we set $i(n):= i_p$. Then, $\depth_{\la,\m}(i(n)) \leq   \depth'_{\la,\m}(i_k) - (k-p) = n$. Yet, it also follows from the definition of the depth function that $\depth_{\la,\m}(i) \leq \depth_{\la,\m}(i(n)) -p$.

\end{proof}

A particular simple corollary is that the equality
\begin{equation}\label{eq:recu}
\depth_{\la,\m}(i)=\min(\depth'_{\la,\m}(i),\{\depth_{\la,\m}(s)-1:s\in I, \Delta_s\smlr\Delta_i\})
\end{equation}
holds, for all $i\in I$.

\begin{lemma}\label{lem:import}
\begin{enumerate}
  \item\label{eq:imp-1}   Let $i,i'\in I$ be indices, such that $e(\Delta_i) = e(\Delta_{i'})$ and $b(\Delta_i)< b(\Delta_{i'}) \leq b(\Delta_{i^\vee})$. Then, $\depth_{\la,\m}(i')= \depth_{\la,\m}(i)$ and $\Delta^{\clubsuit}_{i'}= \Delta_{i'}$.
  \item\label{eq:impmore} Let $i\in J$ and $i'\in I$ be indices, such that $e(\Delta_i) = e(\Delta_{i'})$ and $b(\Delta_{i'}) \leq b(\Delta_{i^\vee})$. Then, $\depth_{\la,\m}(i')= \depth_{\la,\m}(i)$.
  \item\label{eq:imp-2}  Let $i\in I\cup J$ and $i'\in I$ be indices, such that $b(\Delta_i)= b(\Delta_{i'})$ and $e(\Delta_i) < e(\Delta_{i'})\leq e(\Delta_{i_\#})$. Then, $\depth_{\la,\m}(i')= \depth_{\la,\m}(i)$ and $\Delta_{i_\#}= \Delta_{i'}$.
\end{enumerate}

\end{lemma}

\begin{proof}

Once the equality $\depth_{\la,\m}(i')= \depth_{\la,\m}(i)$ is established, the rest of the statements will easily follow.

In cases $\eqref{eq:imp-1}$ and $\eqref{eq:imp-2}$, the inequalities $b(\Delta_i)\leq b(\Delta_{i'})$ and $e(\Delta_i)\leq e(\Delta_{i'})$ imply that $\depth_{\la,\m}(i')\leq \depth_{\la,\m}(i)$. The same is implied by $e(\Delta_{i'}) = e(\Delta_i)$ in case $\eqref{eq:impmore}$.

Fix $j =j_{\depth_{\la,\m}(i)}$. 
Since $b(\Delta_{i^\vee}) \leq b(\Delta_j)$ and $e(\Delta_{i_\#})\leq e(\Delta_j)$ hold, the assumed inequalities imply in all cases that $\lshft\Delta_{i'} \smlr \Delta_j$. Hence, $\depth'_{\la,\m}(i')\geq \depth_{\la,\m}(i)$.

Assume now the contrary, that is, $\depth_{\la,\m}(i')< \depth_{\la,\m}(i)$. Then, by Lemma \ref{lem:techdep}, there is an index $i_1 \in I$, for which $\Delta_{i_1} \smlr \Delta_{i'}$ and $\depth_{\la,\m}(i_1) =\depth_{\la,\m}(i)$.



In cases $\eqref{eq:imp-1}$ and $\eqref{eq:impmore}$, the relations
\[
\left\{\begin{array}{l} e(\Delta_{i_1})< e(\Delta_{i'}) = e(\Delta_i) \\ b(\Delta_{i_1})< b(\Delta_{i'})\leq b(\Delta_{i^\vee}) \end{array}\right.
\]
become a contradiction to the minimality property defining $i^\vee$.

In case $\eqref{eq:imp-2}$, a similar contradiction to the defining property of $i_\#$ is deduced from
\[
\left\{\begin{array}{l} e(\Delta_{i_1})< e(\Delta_{i'})\leq e(\Delta_{i_\#}) \\ b(\Delta_{i_1})< b(\Delta_{i'})= b(\Delta_{i}) \end{array}\right.\;.
\]

\end{proof}

\subsection{RSK-standard modules}

Let $(\la,\m)\in \pairs$ be a permissible pair, and $\n = \Vein (\la,\m)$.

The ladder multisegment $\la = \m(\lambda,\mu)$ gives rise to a homogeneous representation $\Xi(\lambda,\mu)= L_{\la}\in \irr_{\mathcal{D}}$, while $\m,\n$ give rise to $L_{\m}, L_{\n}\in \irr_{\mathcal{D}}$ through the Kleshchev-Ram procedure of Section \ref{sec:kr-class}.

The key result \cite[Theorem 4.3]{gur-lap} may be imported to the quiver Hecke algebra setting in the following form.

\begin{theorem}\label{thm:graded-quot}
  For $(\la,\m)\in \pairs$, there is integer $\widetilde{\Lambda}(\la,\m)\in \mathbb{Z}$, so that $L_{\Vein(\la,\m)}\langle -\widetilde{\Lambda}(\la,\m)\rangle$ is the head of $L_{\la}\circ L_{\m}$.
\end{theorem}
\begin{proof}

The homogeneous module in $L_{\la}$ is known to be square-irreducible (or, real) in the sense that $L_{\la}\circ L_{\la}$ is irreducible (See the argument in the proof of Theorem \ref{thm:mmain}).

Hence, $L_{\la}\circ L_{\m}$ has a simple head $S\langle \widetilde{\Lambda}(\la,\m)\rangle$, for $S\in \irr_{\mathcal{D}}$. (See Section \ref{sec:normal} for further details.)

Fixing any $\rho\in \cusp_m$, we may write $Z(\la),Z(\m),Z(\n)\in \irr^{\mathbb{Z}}_{\rho}$. By \cite[Theorem 4.3]{gur-lap}, $Z(\n)$ is the head of $Z(\la)\times Z(\m)$. Applying $\mathcal{E}_\rho$ and using Proposition \ref{prop:monoid}, we see that $\mathcal{E}_\rho(Z(\n))$ must appear as a quotient module of $(L_{\la}\circ L_{\m})^{\fgt}$.

From the identity \eqref{eq:homsp}, we must have $S^{\fgt}\cong L_{\n}$ in $\widehat{\mathcal{D}}$.

\end{proof}

Consider now any module $L_{\m}\in \irr_{\mathcal{D}}$, given by a multisegment $\m\in \Mult$ of width $\omega = \omega(\m)$. Let
\[
\RSK(\m) = (\la_1,\ldots,\la_{\omega}) \in \Lad^{\omega}
\]
be its RSK-transform.

Taking record of the multisegments $\m = \m_0, \m_1,\ldots, \m_{\omega-1}, \m_\omega = 0 \in \Mult$, so that $\m_{i-1} = \Vien(\la_i,\m_i)$, we set the integer
\[
d(\m) = \widetilde{\Lambda}(\la_1,\m_1) + \ldots + \widetilde{\Lambda}(\la_{\omega-1},\m_{\omega-1})\;.
\]
We define the \textit{RSK-standard} module associated with $\m$ to be
\[
\Gamma(\m):=L_{\la_1}\circ \cdots \circ L_{\la_{\omega(\m)}}\langle -d(\m)\rangle \in R(\wt(\m))-\gmod\;.
\]

\begin{theorem}\label{thm:gurlap}
The self-dual simple module $L_\m$ appears as a quotient module of $\Gamma(\m)$.
\end{theorem}
\begin{proof}
The quotient is produced by a consecutive application of Theorem \ref{thm:graded-quot} on $L_{\la_1}, \ldots,L_{\la_{\omega}}$, taking into the account the definition of the RSK-transform.
\end{proof}

\section{Degree computation}\label{sec:deg-comp}

Let $(\la,\m)\in \pairs$ be a permissible pair as before, and $\n = \Vein(\la,\m)$. We would like to produce an explicit formula for the number $\widetilde{\Lambda}(\la,\m)$ appearing in Theorem \ref{thm:graded-quot}.

The main feature that will facilitate this computation is the following multiplicity-one result.

\begin{proposition}\label{prop:mult1}
The graded multiplicity $m(L_\la\circ L_\m, L_{\n}^{\otimes})(q)\in \mathbb{Z}_{\geq0}[q,q^{-1}]$ is a monomial.

In other words, there is a unique integer $\kappa = \kappa(\la,\m)$, for which $L^{\otimes}_{\n}\langle \kappa \rangle$ appears as a subquotient of $\rres_{\underline{\beta}}(L_\la \circ L_{\m})$, and its ungraded multiplicity in the Jordan-H\"{o}lder series of $\rres_{\underline{\beta}}(L_\la \circ L_{\m})^{\fgt}$ is $1$.
\end{proposition}

The equality
\[
q^{\kappa(\la,\m)} = m(L_\la\circ L_\m, L_{\n}^{\otimes})(q) =
\]
\[
= m(L_\la\circ L_\m, L_{\n})(q)\cdot  m(L_{\n}, L_{\n}^{\otimes})(q) + \sum_{L_{\n}\not\cong M\in \irr(\wt(\n))}  m(L_\la\circ L_\m, M)(q) \cdot m(M, L_{\n}^{\otimes})(q)
\]
of Laurent polynomials follows. Taking Corollary \ref{cor:indic-self} into account, we may write
\[
m(L_{\n}, L_{\n}^{\otimes}) = 1 + P(q)\;,
\]
for $P(q)\in \mathbb{Z}_{\geq0}[q,q^{-1}]$. Hence, from positivity of coefficients involved, we must have $q^{\kappa(\la,\m)} =  m(L_\la\circ L_\m, L_{\n})(q)$ and
\[
-\widetilde{\Lambda}(\la,\m) = \kappa(\la,\m)\;.
\]

We are left with a computation of the degree $\kappa(\la,\m)$, that will be performed in Section \ref{sec:compute-k}.

\subsection{Proof of Proposition \ref{prop:mult1}}

Let us take a choice of segments $\Delta_1, \ldots, \Delta_s\in \Seg$ with $b(\Delta_1) \leq \ldots\leq b(\Delta_s)$, a homogeneous module $\Xi \in \irr$, and a multisegment $0\neq \n \in \Mult$.

We write $\n = \n_1 +\ldots + \n_t$ as a sum of left-aligned multisegments $\n_i\in \Mult$ with $a_i = b(\n_i)$ satisfying $a_1< \ldots < a_t$, as in Section \ref{sec:indic}. We write $\underline{\gamma}= (\wt(\n_1),\ldots, \wt(\n_t))$.

Let us consider the module
\[
\Pi = \Xi\boxtimes L_{\Delta_1}\boxtimes \cdots \boxtimes L_{\Delta_s}\in \irr(\underline{\beta})\;,
\]
for $\underline{\beta} = (\wt(\Xi), \Delta_1, \ldots, \Delta_s)$.

With the notations of Section \ref{sec:mackey}, let $M(\Pi, \n)\subset M(\underline{\beta},\underline{\gamma})$ be subset of $\delta \in  M(\underline{\beta},\underline{\gamma})$ for which $K_\delta(\Pi)\neq 0$ (i.e. $\rres^{\underline{\beta}}_{\underline{\delta}^{row}}(M)\neq 0$). Then, by Proposition \ref{prop:mackey},
\[
[\rres_{\underline{\gamma}}(\iind_{\underline{\beta}}(\Pi))]= \sum_{\delta\in M(\Pi, \n)} [ K_\delta(\Pi)]
\]
holds in the Grothendieck group of $R(\underline{\gamma})-\gmod$.

\begin{lemma}\label{lem:tria}
\begin{enumerate}
  \item\label{tria-it1} Each matrix in $\delta= (\delta_{q,r})_{q=0,\ldots, s\,, r= 1,\ldots, t}\in M(\Pi,\n)$ is triangular, in the sense that for every $1\leq q\leq s$ and $1\leq r\leq t$ satisfying $b(\Delta_{q}) < a_r$, we have $\delta_{q,r}=0$.

  \item\label{tria-it2} Suppose that $\mathfrak{b}(\n) = \alpha_{b(\Delta_1)} + \ldots + \alpha_{b(\Delta_s)} + \mathfrak{b}(\Xi)$.

  Then, for all $\delta= (\delta_{q,r})\in M(\Pi,\n)$ and all $1\leq q\leq s$ and $r$ that satisfy $b(\Delta_{q}) =a_r$, we have $\delta_{q,r}\neq 0$.

\end{enumerate}

\end{lemma}
\begin{proof}
Let us fix $\delta= (\delta_{q,r})\in M(\Pi,\n)$. For a fixed $1\leq q\leq s$, let $r_0$ be the maximal index for which $\delta_{q,r_0}\neq0$. From \eqref{eq:seg-jac}, we know that $\delta_{q,r_0}$ is a segment, with $b(\delta_{q,r_0}) = b(\Delta_{q})$. Thus, $\wt(\n_{r_0})$ contains $\alpha_{b(\Delta_{q})}$, implying $a_{r_0}\leq b(\Delta_{q})$, which proves \eqref{tria-it1}.

Suppose now that the assumption in \eqref{tria-it2} holds. Let $1\leq r\leq t$ be fixed, and write
\[
\{q,q+1,\ldots, q+m-1\} = \{q\;: b(\Delta_{q}) = a_r\}\;.
\]
Note, that the height $|\n_r|$ must equal either $m$ or $m+1$ by assumption.

Recall that $\wt(\n_r) = \delta_{1,r} +\ldots + \delta_{s+1,r}$. Considering \eqref{tria-it1}, we see that $|\n_r|$ is the multiplicity of $\alpha_{a_r}$ in $\delta_{q,r} + \delta_{q+1,r}+ \ldots + \delta_{q+m-1,r} + \delta_{s+1,r}$.

In case $\delta_{0,r}$ does not contain $\alpha_{a_r}$, all $m$ segments $\delta_{q,r}, \delta_{q+1,r},\ldots, \delta_{q+m-1,r}$ must be non-zero.

Otherwise, let us recall that $\rres_{(\delta_{0, 1}, \delta_{0,2},\ldots , \delta_{0,t})}^{\underline{\beta}}(\Xi)\neq 0$ is homogeneous. Since $\delta\in M(\Pi,\n)$, we also know that $a\geq a_j$, for all $\alpha_a\leq \delta_{s+1,j}$.

Considering the description of possible non-zero restrictions of homogeneous modules in Proposition \ref{prop:ladderjac}, we first deduce that the multiplicity of $\alpha_{a_r}$ in $\delta_{0,r}$ is $1$. Moreover, $a_r$ must appear as a begin point of a segment in the multisegment $\la\in \Lad$, for which $\Xi= L_{\la}$. This implies $|\n_r| = m+1$, which brings us again to the implication that each one of the segments $\delta_{q,r}, \delta_{q+1,r},\ldots, \delta_{q+m-1,r}$ must be non-zero.

\end{proof}

With the same notations, suppose now that the multisegment $\m = \Delta_1 + \ldots + \Delta_s$ and the ladder multisegment $\la = \Delta_{j_1} + \ldots + \Delta_{j_t}\in \Lad$, for which $\Xi = L_{\la}$, constitute a permissible pair $(\la,\m)\in \pairs$. Suppose further that $\n = \Vein(\la,\m)$.

We may assume that $I = \{1,\ldots,s\}$ and $J =\{ j_1,\ldots,j_t\}$ satisfy the assumptions of Section \ref{sec:algo} (in particular $I$ is ordered according to condition \eqref{eq:ordered}), and adopt the notations of that section for the RSK algorithm.

Recall that $\mathfrak{b}(\n) = \mathfrak{b}(\m) + \mathfrak{b}(\la)$.

In particular, we may write $\m = \m_1 + \ldots + \m_t$, with either $\m_r=0$ or $\m_r$ a left-aligned multisegment, such that $b(\m_r) = a_r$.

For every $1\leq q\leq s$, let $r(q)$ be the unique index satisfying $a_{r(q)} = b(\Delta_{q})$, that is,
\[
\m_r = \sum_{q\,:\, r(q) = r} \Delta_q\;.
\]

\begin{proposition}\label{prop:heavy}

There is a unique $\delta= (\delta_{r,q})\in  M(\Pi,\n)$, for which the graded multiplicity $m(K_\delta(\Pi), L_{\n}^{\otimes})(q)$ is non-zero.

Moreover, the Laurent polynomial $m(\iind_{\underline{\beta}}(\Pi), L_{\n}^{\otimes})(q)= m(K_\delta(\Pi), L_{\n}^{\otimes})(q)$, for that unique $\delta$, is a monomial.

Concretely, we have
\[
K_\delta(\Pi) = ( \Xi_1\circ L_{\m_1}) \boxtimes \cdots  \boxtimes (\Xi_t \circ L_{\m_t} )\langle N\rangle \;,
\]
where
\[
\rres_{(\delta_{0, 1}, \delta_{0,2},\ldots , \delta_{0,t})}^{\underline{\beta}}(\Xi) = \Xi_1 \boxtimes \cdots \boxtimes  \Xi_t\;,
\]
and $N = \deg(\delta) - {|\m_1| \choose 2} - \ldots -{ |\m_t|  \choose 2}$ .

\end{proposition}
\begin{proof}
It follows from Theorem \ref{thm:graded-quot}, Theorem \ref{thm:kr} and Corollary \ref{cor:indic-self}, that $m(\Xi\circ S(\m), L_{\n}^{\otimes})(q)$, and hence, $m(\iind_{\underline{\beta}}(\Pi), L_{\n}^{\otimes})(q)$, are non-zero.

By Proposition \ref{prop:mackey}, it follows that $\delta\in M(\Pi,\n)$ with $m(K_\delta(\Pi), L_{\n}^{\otimes})(q)\neq 0$, exists. We are left to prove its uniqueness.

Suppose that $\delta= (\delta_{q,r})\in  M(\Pi,\n)$ with non-zero graded multiplicity $m(K_\delta(\Pi), L_{\n}^{\otimes})(q)$ is given, and $\Xi_1, \ldots, \Xi_t\in \irr_{\mathcal{D}}$ are as in the statement.

By formula \eqref{eq:seg-jac}, we know that for each $1\leq q\leq s$, $\delta_{q,r(q)} = \Delta'_q\in \Seg$ is a segment with $b(\Delta'_q) = b(\Delta_q)= a_{r(q)} $ and $e(\Delta'_q)\leq e(\Delta_{q})$.

We will show that $\delta_{q,r}= 0$, for all $1\leq  q\leq s$ and $r\neq r(q)$. In other words, we claim that for all $1\leq q \leq s$,  $\Delta_q= \Delta'_q$.

Uniqueness will then follow, since $\delta_{0,r}= \wt(\n_r) - \wt(\m_r)$ will be imposed, for all $1\leq r\leq t$.


Assume the contrary, that is, $\delta_{q_0,r_0}= \Delta$ with $r_0\neq r(q_0)$, is non-trivial.

We may select $1\leq q_0\leq s$ to be the maximal index for which such $r_0$ exists, and $r_0$ to be the maximal such index. Again, by Lemma \ref{lem:tria} we know that $r_0< r(q_0)$.

We write $x = b(\Delta)$ and $y= e(\Delta_{q_0})$. Then, $x\leq e(\Delta)\leq y$.


Note, that $e(\Delta'_{q_0}) = x-1<y$. Hence,
\begin{equation}\label{eq:ineq1}
\alpha_y \leq \wt(\m_{r(q_0)})- \left(\sum_{q\,:\,r(q) = r(q_0)}\Delta'_{q}\right) = \wt(\m_{r(q_0)})- \left(\sum_{q\,:\,r(q) = r(q_0)}\delta_{q,r(q)}\right).
\end{equation}


Now, let us consider the index $q_1= (q_0)^\vee\in I\cup J$. We set $a_{r_1} = b(\Delta_{q_1})$ (that is, $r_1= r(q_1)$, in case $q_1\in I$.)


Recall that $e(\Delta^{\clubsuit}_{q_1}) = e(\Delta_{q_0})= y$. If $q_1\in I$, we have $e(\Delta_{q_1})<y$. Otherwise $q_1\in J$ and $\Delta^{\clubsuit}_{q_1} \leq \wt(\n_{r_1}) - \wt( \m_{r_1})$. In either case, we obtain
\begin{equation}\label{eq:ineq2}
\alpha_y \leq \wt(\n_{r_1}) - \wt( \m_{r_1})\;.
\end{equation}


By definition of the $i\mapsto i^\vee$ operation, we know that $r(q_0)\leq r(q_1)$. It follows from maximality of $q_0$ that
\begin{equation}\label{eq:ineq3}
\wt(\n_{r_1}) = \delta_{0,r_1}+\sum_{q\,:\,r(q) = r_1}\delta_{q,r(q)}\;.
\end{equation}

Now, the spherical module $L_{\n_{r_1}}$, up to a shift of grading, was assumed to appear as a subquotient of a product involving $\Xi_{r_1}$. By Proposition \ref{prop:sph}, that means that $\Xi_{r_1}$ must be a spherical module.

Recalling the description in Proposition \ref{prop:ladderjac}, we see that means $\Xi_{r_1} = L_{\la'}$, for a multisegment $\la' = \Delta^\circ_1 + \ldots + \Delta^{\circ}_p\in \Lad$, such that $e(\Delta^\circ_{i+1})+1< b(\Delta^\circ_{i})$, for all $1\leq i \leq p-1$.

Considering \eqref{eq:ineq1},\eqref{eq:ineq2},\eqref{eq:ineq3}, we see that the equality $r(q_0)= r_1$ would have implied $2\alpha_y \leq \delta_{s+1,r_1}= \wt(\la')$, which is a contradiction.

Hence, $r(q_0)< r_1$ and by maximality of $q_0$, we have $\Delta'_q = \Delta_q$, for all $q$ with $r(q) = r_1$. In particular, as an equation in $Q_+$ (as opposed to $\Mult$), we have
\[
\wt(\la') = \wt(\n_{r_1}) - \wt( \m_{r_1}) = \left\{\begin{array}{ll} \Delta^{\clubsuit}_j + \sum_{q\,:\,r(q) = r_1} \Delta(e(\Delta_q)+1, e(\Delta^{\clubsuit}_q))  \;,& \exists j\in J,\; b(\Delta_j) = a_{r_1} \\
 \sum_{q\,:\,r(q) = r_1} \Delta(e(\Delta_q)+1, e(\Delta^{\clubsuit}_q))
 \end{array} \right.\;.
\]

The containment \eqref{eq:ineq2} implies $\alpha_y\leq \wt(\la')$. Let $p_0$ be the index for which $\alpha_y\leq \Delta^\circ_{p_0}$.

Suppose first that $y < e(\Delta^\circ_{p_0})$. Noting again that $y= e(\Delta^{\clubsuit}_{q_1})$, we see that such situation can happen only when $y= e(\Delta_{q_2}) < e(\Delta^{\clubsuit}_{q_2})$, for an index $q_2$ with $r(q_2) =r_1$. This is a contradiction to Lemma \ref{lem:import}\eqref{eq:imp-1}.

Otherwise, $y=e(\Delta^{\circ}_{p_0})$. From the containment \eqref{eq:ineq1}, we also have $\alpha_y\leq \wt(\la'')$, where $\Xi_{r(q_0)} = L_{\la''}$. Using Proposition \ref{prop:ladderjac} again, these facts point on an existence of an index $r(q_0)< r_2<r_1$, for which $\alpha_{y+1} \leq B(\la''')$, where $\Xi_{r_2} = L_{\la'''}$. Reasoning as in the previous case, we see that there must be an index $q_2$ with $r(q_2) = r_2$ and $y = e(\Delta_{q_2}) < e(\Delta^{\clubsuit}_{q_2})$. This gives a contradiction to Lemma \ref{lem:import}\eqref{eq:imp-1}.

The formula for $N$ follows from the fact that, for all $1\leq r\leq t$,
\[
L_{\m_r} = \nabla(\m_r) = \circ_{q\,:\, r(q) = r} L_{\Delta_q} \left\langle { |\m_r| \choose 2} \right\rangle\;,
\]
where the product is taken with indices in ascending order.

Finally, the fact that $K_\delta(\Pi)$ is a monomial follows again from Proposition \ref{prop:sph}.

\end{proof}

Note, that $\Xi\circ L_\m$ is a quotient module of $ \Xi \circ \KR(\m)$, which, up to a shift of grading, equals to $ \Xi \circ \iind_{\underline{\beta}}(\Pi)$. By Proposition \ref{prop:heavy} we see that $m(\Xi\circ L_\m, L_{\n}^{\otimes})$ must be a monomial (non-vanishing follows from Theorem \ref{thm:graded-quot}, as in the proof of Proposition \ref{prop:heavy}).

Proposition \ref{prop:mult1} is now a consequence of the Kleshchev-Ram construction (Theorem \ref{thm:kr}).

\subsection{Computing $\kappa(\la,\m)$}\label{sec:compute-k}

Given $(\la,\m)\in \pairs$ and $\n = \Vein (\la,\m)$ with all previous associated notations in place, we would like to obtain a combinatorial formula for the degree
\[
\kappa = \kappa(\la,\m)\;,
\]
for which $L_{\n}\langle \kappa \rangle$ appears as a subquotient of the restriction of $L_{\la} \circ L_{\m}$.


Let $\delta= (\delta_{q,r})\in M(\Pi,\n)$ be the matrix supplied by Proposition \ref{prop:heavy}, and $\Xi_1, \ldots,\Xi_t\in \irr_{\mathcal{D}}$ the homogeneous modules as in the statement of the proposition.

For all $1\leq r\leq t$, let $n_r$ be the degree for which $L_{\n_r}\langle n_r\rangle $ appears in $\Xi_r\circ L_{\m_r}$.

Let us recall that the parameterization
\[
\m = \sum_{i\in I}\Delta_i
\]
is still given with the index set $I$ taken according to condition \eqref{eq:ordered}.

By Lemma \ref{lem:kr-indi}, we know that
\[
\KR(\m)  \left\langle {|\m_1| \choose 2} + \ldots +{ |\m_t|  \choose 2}\right\rangle \cong  \iind_{\underline{\beta}}(\Pi)\;,
\]
and that $m(\KR(\m), L_{\m})(q)=1$. Hence, by Proposition \ref{prop:heavy},
\begin{equation}\label{eq:kappa}
\kappa = n_1 + \ldots + n_t + \deg(\delta)\;,
\end{equation}
where $\deg(\delta)$ is as in formula \eqref{eq:deg-form}.

\begin{lemma}\label{lem:degseg}
Let $\nabla_1 = \nabla(a\,;b_1,\ldots,b_k), \nabla_2 = \nabla(a\,;b'_1,\ldots,b'_l)$ be given graded simple modules (as defined in Section \ref{sec:indic}), and $L = L_{\Delta}\in \irr_{\mathcal{D}}$ a given segment module, for which $\nabla_2\langle m\rangle$ appears as a subquotient in $L \circ \nabla_1$.

Then, its degree is given as

\[
m = -\#\{ 1\leq i\leq k\;:\; b(\Delta) < b_i\leq e(\Delta)\}\;.
\]

\end{lemma}

\begin{proof}
We write $\Delta = \Delta(c,d)$. Let us denote the set of indices
\[
A = \{ 1\leq i\leq k\;:\; c< b_i\leq d\}\;,
\]
and $i_1 = \min A$, if $A\not= \emptyset$, or $i_1=k+1$ otherwise.

If $c-1\not\in \{b_1,\ldots,b_k\}$, considering the equality
\[
\Delta = \wt(\m(a\,;b'_1,\ldots,b'_l))- \wt(\m(a\,;b_1,\ldots,b_k))\;,
\]
we must have $c=a$, $A=\{i_1 \leq i\leq k\}$ and $l=k+1$. In this case, by Lemma \ref{lem:unlink-seg} we have

\[
L\circ \nabla_1 \cong L_{\Delta(a,b_1)}\circ\cdots\circ L_{\Delta(a,b_{i_1-1})} \circ L \circ  L_{\Delta(a,b_{i_1})} \circ \cdots \circ  L_{\Delta(a,b_k)} \langle {k \choose 2} + i_1-1\rangle \cong
\]

\[
\cong \nabla_2\langle -k + i_1-1\rangle\;.
\]
Since $|A| = k- i_1+1$,  the statement follows.

Otherwise, let $i_2$ be the smallest index for which $b_{i_2} = c-1$. We have $A = \{i_1\leq i < i_2\}$. Then, by Lemma \ref{lem:unlink-seg},
\[
L\circ \nabla_1 \cong  L_{\Delta(a,b_1)}\circ\cdots\circ L_{\Delta(a,b_{i_1-1})} \circ L \circ  L_{\Delta(a,b_{i_1})} \circ \cdots \circ  L_{\Delta(a,b_k)} \langle {k \choose 2} \rangle \cong
\]
\[
\cong L_{\Delta(a,b_1)}\circ\cdots\circ L_{\Delta(a,b_{i_1-1})} \circ L \circ L_{\Delta(a,b_{i_2})} \circ   L_{\Delta(a,b_{i_1})} \circ \cdots \circ  L_{\Delta(a,b_{i_2-1})} \circ  L_{\Delta(a,b_{i_2+1})} \circ \cdots \circ  L_{\Delta(a,b_k)}\langle {k \choose 2} - |I|\rangle\;.
\]
Recall now (formula \eqref{eq:seg-jac}), that $L\boxtimes L_{\Delta(a,b_{i_2})}$ appears as a subrepresentation of the restriction of $L_{\Delta(a,d)}$. By adjunction, that means that $L_{\Delta(a,d)}$ appears as a quotient of $L\circ L_{\Delta(a,b_{i_2})}$. Hence, from exactness of the convolution product, $\nabla_2\langle - |A|\rangle$ appears in $L\circ \nabla_1 $.

\end{proof}

Let us now attach a few easily computable integer numbers to the permissible pair $(\la,\m)\in\pairs$.

First, for any pair of multisegments
\[
\m_1= \sum_{i=1}^{k_1} \Delta^1_i,\quad \m_2= \sum_{i=1}^{k_2} \Delta^2_i\in \Mult\;,
\]
we define the number
\[
C(\m_1,\m_2) = \# \{(i,j)\;:\; b(\Delta^1_i) = e(\Delta^2_j)+1 \}\;.
\]

We recall that $\la= \sum_{j\in J} \Delta_j\in \Lad$ is given with $J =\{ j_1,\ldots,j_t\}$, so that all assumptions and notations of the algorithm in Section \ref{sec:algo} are in place.

In particular, we recall that,
\[
\n= \Vein(\la,\m) = \sum_{i\in I\cup J} \Delta^{\clubsuit}_i\;.
\]
We keep track of the following sets of indices:
\[
\begin{array}{l}  \nu(\la,\m)_1 = \{(i_1,i_2)\in I\times I\;:\; b(\Delta_{i_1}) \leq b(\Delta_{i_2}),\, e(\Delta_{i_2})<e(\Delta^{\clubsuit}_{i_2}) = e(\Delta_{i_1})\} \;,\\
 \nu(\la,\m)_2 = \{(i_1,i_2)\in I\times I\;:\; b(\Delta_{i_1}) < b(\Delta_{i_2}),\,  e(\Delta_{i_1})= e(\Delta_{i_2})<e(\Delta^{\clubsuit}_{i_2})\}  \;,\\
   \nu(\la,\m)_3 = \{(i_1,j)\in I\times J\;:\; b(\Delta_{i_1}) \leq b(\Delta_{j}),\, e(\Delta^{\clubsuit}_{j}) = e(\Delta_{i_1})\}  \;.
  \end{array}
\]
In these terms, we set the integer
\[
D(\la,\m) = |\nu(\la,\m)_1| - |\nu(\la,\m)_2| + |\nu(\la,\m)_3|\;.
\]

\begin{proposition}\label{prop:CD}
For a permissible pair $(\la,\m)\in \pairs$, we have $\kappa(\la,\m)= C(\la,\m) - D(\la,\m)$.
\end{proposition}

\begin{proof}

Let us decompose $I = \cup_{r=1}^t I_r$, so that $I_r = \{q\in I\;:\;r(q) = r\}$ and $\m_r = \sum_{i\in I_r} \Delta_i$.
Then,
\[
\n_r =  \sum_{i\in I_r} \Delta^{\clubsuit}_i + \Delta'_r\;,
\]
where $\Delta'_r = \Delta^{\clubsuit}_{j}$, if there exists $j\in J$, such that $b(\Delta_{j}) = a_r$. If such $j_r$ does not exist, we treat $\Delta'_r$ as an empty segment (i.e. $0\in Q_+$).

Thus, for all $1\leq r\leq t$,

\begin{equation}\label{eq:sum-supp}
\wt(\Xi_r) = \wt(\n_r) - \wt(\m_r) = \sum_{q\in I_r} \Delta(e(\Delta_q)+1, e(\Delta^{\clubsuit}_q)) + \Delta'_r\in Q_+\;.
\end{equation}
By the description of Proposition \ref{prop:heavy}, we have
\[
\deg(\delta) = - \sum_{1\leq r_1 < r_2\leq t} (\wt(\Xi_{r_2}), \wt(\m_{r_1}))\;.
\]

We can write $\deg(\delta) = d_1 + d_2$ with
\[
d_1 = - \sum_{1\leq r_1 < r_2\leq t}  \sum_{q_1\in I_{r_1},\, q_2\in I_{r_2}} (\Delta_{q_1}, \Delta(e(\Delta_{q_2})+1, e(\Delta^{\clubsuit}_{q_2})))\;,
\]
\[
d_2 = - \sum_{1\leq r_1 < r_2\leq t}  \sum_{q_1\in I_{r_1}} (\Delta_{q_1}, \Delta'_{r_2}))\;.
\]

A straightforward computation (see, for example, \cite[Lemma 5.2]{me-decomp}) shows now that
\[
d_1 = -|\nu(\la,\m)_{1,<}| + |\nu(\la,\m)_2|\;,
\]
where
\[
\nu(\la,\m)_{1,<} = \nu(\la,\m)_{1} \;\cap\; \{(i,i')\in I\times I \;:\; b(\Delta_i)<b(\Delta_{i'})\}\;.
\]
Similarly,
\[
d_2 = -|\nu(\la,\m)_{3,<}| + C(\la, \m) \;,
\]
where
\[ \nu(\la,\m)_{3,<} = \nu(\la,\m)_{3} \;\cap\; \{(i,i')\in I\times I \;:\; b(\Delta_i)<b(\Delta_{i'})\}\;.
\]

By the identity \eqref{eq:kappa}, it is left to show that
\begin{equation}\label{eq:mus}
n_1 +\ldots + n_t  = -|\nu(\la,\m)_{1}\setminus \nu(\la,\m)_{1,<}| -  |\nu(\la,\m)_{3}\setminus\nu(\la,\m)_{3,<}|\;.
\end{equation}

Let us fix $1\leq r\leq t$. We have seen in the proof of Proposition \ref{prop:heavy} that $\Xi_r\in \irr_{\mathcal{D}}$ must be a spherical module. It then follows, as in that proof, that
\[
\Xi_r = L_{\Delta^\circ_1}\circ \cdots \circ L_{\Delta^\circ_p}\;,
\]
for segments $\Delta^{\circ}_1,\ldots,\Delta^\circ_p\in \Seg$, with $e(\Delta^\circ_{i+1})+1< b(\Delta^\circ_i)$, for all $1\leq i \leq p$.

Hence, we are left with determining the degree $n_r$, for which $L_{\n_r}\langle n_r \rangle$ appears as a subquotient in
\[
L_{\Delta^\circ_1}\circ \cdots \circ L_{\Delta^\circ_p} \circ L_{\m_r}\;.
\]

Taking \eqref{eq:sum-supp} into account and computing successively with Lemma \ref{lem:degseg}, we obtain
\begin{equation}\label{eq:nform}
n_r =n'_r- \sum_{q_1\in I_r} \#\{ q_2\in I_r\;:\; e(\Delta_{q_1}) < e(\Delta_{q_2}) \leq e(\Delta^{\clubsuit}_{q_1})\} \;,
\end{equation}
where
\begin{equation}\label{eq:nform2}
n'_r = \left\{\begin{array}{ll}   -\#\{ q\in I_r\;:\; e(\Delta_{q}) \leq e(\Delta'_r)\} &  \Delta'_r\neq 0 \\  0 &  \Delta'_r =0\end{array} \right.\;.
\end{equation}

Finally, we invoke Lemma \ref{lem:import}\eqref{eq:imp-2} to show that the weak inequalities in the formulas \eqref{eq:nform} and \eqref{eq:nform2} can in fact be written as equalities. Hence, equation \eqref{eq:mus} clearly follows.

\end{proof}

\begin{proposition}
For a permissible pair $(\la,\m)\in \pairs$, we have
\[
D(\la,\m)= |\m|\;.
\]
\end{proposition}

\begin{proof}
For $i\in I$, we write
\[
\nu(\la,\m)^i_1 = \nu(\la,\m)_1 \cap \left(\{i\}\times I\right),\; \nu(\la,\m)^i_2 = \nu(\la,\m)_2 \cap \left(\{i\}\times I\right),\; \nu(\la,\m)^i_3 = \nu(\la,\m)_1 \cap \left(\{i\}\times J\right)\;,
\]
and $D_i(\la,\m) = |\nu(\la,\m)^i_1| - |\nu(\la,\m)^i_2| + |\nu(\la,\m)^i_3|$. Hence, $D(\la,\m) = \sum_{i\in I} D_i(\la,\m)$.

We will prove that $D_i(\la,\m)=1$, for each $i\in I$. Let us fix $i_0\in I$ for the rest of the proof.

Let $I\cup J = \cup_{\alpha\in \mathcal{A}}  C_\alpha$ be the cycle decomposition of the set $I\cup J$ relative to the permutation defining $\n=\Vein(\la,\m)$. In particular, we set $i_0\in C_{\alpha_0}$.

We further divide $\nu(\la,\m)^{i_0}_u$ into the disjoint subsets $\nu(\la,\m)^{i_0,\alpha}_u = \nu(\la,\m)^{i_0}_u\cap \left(\{i_0\}\times C_\alpha\right)$, for $u=1,2,3$, and similarly write $D_{i_0}(\la,\m) = \sum_{\alpha\in C_\alpha} D_{i_0, \alpha}(\la,\m)$, for
\[
D_{i_0, \alpha}(\la,\m) =|\nu(\la,\m)^{i_0,\alpha}_1| - |\nu(\la,\m)^{i_0,\alpha}_2| + |\nu(\la,\m)^{i_0,\alpha}_3|\;.
\]

Let us first consider $D_{i_0,\alpha_0}(\la,\m)$. We may write $C_{\alpha_0} = \{i^1,\ldots,i^k, j_0\}$, with $i^1 > \ldots > i^k$ indices in $I$ and $j_0\in J$.

Recall that $b(\Delta_{i^1})\geq  \ldots \geq b(\Delta_{i^s})$, and that $e(\Delta_{i^r})\leq e(\Delta^{\clubsuit}_{i^r}) =e(\Delta_{i^{r+1}})$, for $1\leq r\leq k-1$. A particular consequence is that $\nu(\la,\m)^{i_0,\alpha_0}_2$ must be empty.

Moreover, when $(i_0)^\vee\in I$, we easily see that $\nu(\la,\m)^{i_0,\alpha}_1 = \{(i_0)^\vee\}$, while $\nu(\la,\m)^{i_0,\alpha_0}_3$ is empty, as a consequence of $e(\Delta^{\clubsuit}_{j_0}) \leq e(\Delta_{(i_0)^\vee})< e(\Delta_{i_0})$.

Otherwise, $(i_0)^\vee=j_0\in J$ would mean that $\nu(\la,\m)^{i_0,\alpha_0}_1$ is empty, while $\nu(\la,\m)^{i_0,\alpha_0}_3=\{j_0\}$. Either way, we obtain that $D_{i_0,\alpha_0}(\la,\m)=1$.

Let $\alpha_0\neq \alpha\in \mathcal{A}$ be fixed. We are left to show that
\begin{equation}\label{eq-desi}
D_{i_0,\alpha}(\la,\m)=0\;.
\end{equation}

We write again $C_{\alpha} = \{i_\alpha^1,\ldots,i_\alpha^{k_\alpha}, j_\alpha\}$, with $i_\alpha^1 > \ldots > i_\alpha^{k_\alpha}$ indices in $I$ and $j_\alpha=i_\alpha^0\in J$.  Let $r_2$ be the maximal index for which $b(\Delta_{i_0}) \leq b(\Delta_{i_\alpha^{r_2}})$ holds, or $r_2=-1$ if no such index exists.

In case that $e(\Delta_{i_\alpha^r})\neq e(\Delta_{i_0})$ holds, for all $1\leq r\leq r_2+1$, it is evident that $\nu(\la,\m)^{i_0,\alpha}_u,\,u=1,2,3$ are all empty, and \eqref{eq-desi} follows.



Otherwise, $r_2\geq 0$ and we can set $1\leq r_1\leq r_2+1\leq r_3$ to be the indices which satisfy
\[
\{r\,:\, r_1\leq r\leq r_3\} = \{r\,:\, e(\Delta_{i_\alpha^r}) = e(\Delta_{i_0})\}\;.
\]
We first claim that $b(\Delta_{i_0})\leq b(\Delta_{i_\alpha^{r_3}})$. Indeed, a reversed inequality would imply a contradiction to Lemma \ref{lem:import}\eqref{eq:imp-1} and the fact that $i_0\not\in C_\alpha$, since either $(i_\alpha^{r_3})^\vee = j_\alpha$ or $i_\alpha^{r_1}< (i_\alpha^{r_3})^\vee $, and
\[
b(\Delta_{i_\alpha^{r_3}}) < b(\Delta_{i_0}) \leq b(\Delta_{i_\alpha^{r_2}}) \leq b(\Delta_{i_\alpha^{r_1-1}}) \leq b(\Delta_{(i_\alpha^{r_3})^\vee})\;.
\]

We then see that $\nu(\la,\m)^{i_0,\alpha}_1\, \cup \,\nu(\la,\m)^{i_0,\alpha}_3= \{i_\alpha^{r_1-1}\}$.
In particular, $|\nu(\la,\m)^{i_0,\alpha}_1| + |\nu(\la,\m)^{i_0,\alpha}_3|=1$. To reach \eqref{eq-desi}, we are left to show that $|\nu(\la,\m)^{i_0,\alpha}_2|=1$.

If $r_3 < k_\alpha$ or $e(\Delta_{i_\alpha^{k_\alpha}}) < e(\Delta_{j_\alpha})$ hold, we see that $\nu(\la,\m)^{i_0,\alpha}_2 = \{ i^{r_3}\}$.

Finally, let us assume that $r_3 = k_\alpha$ and that
\[
 e(\Delta_{j_\alpha}) = e(\Delta_{i_\alpha^{k_\alpha}}) =  e(\Delta_{i_0})\;.
\]
This is where Lemma \ref{lem:import}\eqref{eq:impmore} is invoked to give a contradiction, since, as before, $b(\Delta_{i_0}) \leq b(\Delta_{(i_\alpha^{r_3})^\vee})$.

\end{proof}

Let us summarize the insight on RSK-standard modules obtained through the computation of this section.

\begin{corollary}\label{cor:maincor}
  For any $(\la,\m)\in \pairs$, the integer $\widetilde{\Lambda}(\la,\m)$ appearing in Theorem \ref{thm:graded-quot} satisfies
   \[
 - \widetilde{\Lambda}(\la,\m)= C(\la,\m) - |\m|\;.
   \]
\end{corollary}

\section{Normal sequences}\label{sec:normal}

Let us recall the theory of normal sequences, as developed in \cite{kk19} and \cite{kkko-mon}.

For any non-zero modules $M_1,M_2\in \mathcal{D}$, there is a well defined non-zero \textit{R-matrix} map
\[
R_{M_1,M_2}:M_1\circ M_2 \to M_2\circ M_1\;,
\]
which is an intertwiner of ungraded quiver Hecke algebra modules. In fact, there is a unique integer $\Lambda = \Lambda(M_1,M_2)\in \mathbb{Z}$, for which
\[
R_{M_1,M_2}: M_1\circ M_2 \to M_2\circ M_1\langle  \Lambda \rangle\;,
\]
is an intertwiner (of graded modules) in $\mathcal{D}$.

It is evident that $\Lambda(M_1\langle k_1 \rangle , M_2\langle k_2 \rangle) = \Lambda(M_1,M_2)$ holds, for any shifts of degree $k_1,k_2\in\mathbb{Z}$.

Let us further recall some properties of the invariant $\Lambda$.

\begin{proposition}\label{prop:prop-kkko}

\begin{enumerate}
  \item\label{it-kk1} \cite[Lemma 3.1.5]{kkko-mon} For $M\in \irr_{\mathcal{D}}$ and any non-zero modules $N_1,\ldots N_k\in \mathcal{D}$, we have
\[
\Lambda(M, N_1\circ \cdots\circ N_k) = \Lambda(M, N_1) +\ldots + \Lambda(M, N_k)\;.
\]

  \item\label{it-kk2} \cite[Proposition 3.2.8]{kkko-mon} For $M_1,M_2 \in \mathcal{D}$ and any subquotient module $N$ of $M_2$, we have
  \[
   \Lambda(M_1, M_2) \geq  \Lambda(M_1, N)\;.
  \]

\end{enumerate}
\end{proposition}

A simple module $L\in \irr_{\mathcal{D}}$ is said to be \textit{square-irreducible} (or \textit{real}), if $L\circ L$ is a simple module.

\begin{proposition}\label{prop:homo-sq}
  Homogeneous modules in $\irr_{\mathcal{D}}$ are square-irreducible.
\end{proposition}
\begin{proof}
The analogous fact for representations of $p$-adic groups is known by \cite[Proposition 5.15]{LM2} (or as part of the general criteria for square-irreduciblity in \cite{LM3}). The statement follows from an application of a functor of the form $\mathcal{E}_\rho$.
\end{proof}

Given a square-irreducible module $L$ and any $M\in \irr_{\mathcal{D}}$, the product $L\circ M$ has a simple head (unique simple quotient module) (\cite[Theorem 3.2]{kkko0}), whose isomorphism class is given as $N\langle -\widetilde{\Lambda}(L,M)\rangle$, for a self-dual $N\in \irr_{\mathcal{D}}$, and an integer $\widetilde{\Lambda}(L,M)\in \mathbb{Z}$.

In light of Proposition \ref{prop:homo-sq}, this definition of $\widetilde{\Lambda}(L,M)$ generalizes the notion defined in Theorem \ref{thm:graded-quot}.

For $(L,M)$ as above, the identity (\cite[Lemma 3.1.4]{kkko-mon})
\begin{equation}\label{eq:linrel}
\Lambda(L,M) = 2\widetilde{\Lambda}(L,M) - (\wt(L),\wt(M))
\end{equation}
holds.

Following \cite{kk19}, we say that a tuple $(L_1,\ldots, L_k)$ of square-irreducible modules in $\irr_{\mathcal{D}}$ is a \textit{normal sequence}, if the composition of R-matrices
\[
R_{L_1,\ldots,L_k} := R_{L_{k-1},L_k}\circ\cdots\circ (R_{L_2,L_k} \circ \cdots \circ R_{L_2,L_3})\circ (R_{L_1,L_k} \circ \cdots \circ R_{L_1,L_2})
\]
does not vanish (as a map from the space of $L_1\circ \cdots \circ L_k$ to that of $L_k\circ\cdots \circ L_1$).

\begin{proposition}\label{prop:simplehd}\cite[Lemma 2.6]{kk19}
For a normal sequence $(L_1,\ldots, L_k)$, the product $L_1\circ \cdots \circ L_k$ has a simple head, given by the image of $R_{L_1,\ldots,L_k}$.
\end{proposition}

\begin{proposition}\label{prop:cond-equiv}
Let $(L_1,\ldots, L_k)$ be a tuple of square-irreducible modules in $\irr_{\mathcal{D}}$.

The following are equivalent:
\begin{enumerate}
  \item\label{itnorm1} The tuple $(L_1,\ldots, L_k)$ is a normal sequence.
  \item\label{itnorm2} The tuple $(L_2,\ldots, L_k)$ is a normal sequence, and the identity
  \[
    \Lambda(L_1, H) = \Lambda(L_1,L_2) + \ldots  + \Lambda(L_1,L_k)
  \]
  holds, for $H\in \irr_{\mathcal{D}}$, such that $H\langle h\rangle$ is the simple head of $L_2\circ\cdots\circ L_k$.
  \item\label{itnorm3} The tuple $(L_2,\ldots, L_k)$ is a normal sequence, and the identity
  \[
    \widetilde{\Lambda}(L_1, H) = \widetilde{\Lambda}(L_1,L_2) + \ldots  + \widetilde{\Lambda}(L_1,L_k)
  \]
  holds, for $H\in \irr_{\mathcal{D}}$, such that $H\langle h\rangle$ is the simple head of $L_2\circ\cdots\circ L_k$.

\end{enumerate}

\end{proposition}
\begin{proof}
The equivalence of \eqref{itnorm1} and \eqref{itnorm2} is \cite[Lemma 2.7]{kk19}, together with the observation that $ \Lambda(L_1, H) = \Lambda(L_1, H\langle h\rangle)$.

Conditions \eqref{itnorm2} and \eqref{itnorm3} are equivalent because of the relation \eqref{eq:linrel}, the fact that
\[
\wt(H) = \wt(L_2 \circ \cdots \circ L_k)  = \wt(L_2) + \ldots + \wt(L_k)\;,
\]
and linearity of the form $(\,,\,)$ on $Q$.

\end{proof}

\begin{corollary}\label{cor:sumform}
  For a normal sequence $(L_1,\ldots, L_k)$, let $H\langle h\rangle$ be the simple head of $L_1\circ \cdots \circ L_k$ with $H\in \irr_{\mathcal{D}}$. Then,
  \[
  h=  -\sum_{1\leq i < j \leq k} \widetilde{\Lambda}(L_i,L_j)\;.
  \]
\end{corollary}
\begin{proof}
  The formula for $h$ follows by inductively applying the identity in condition \eqref{itnorm3} of Proposition \ref{prop:cond-equiv}.
\end{proof}

Let us recall the favorable behavior of the invariant $\Lambda$ in the square-irreducible case.

Note first, that $\Lambda(L,L) = 0$ for square-irreducible $L \in \irr_{\mathcal{D}}$.

\begin{lemma}\label{lem:kkkosqr}\cite[Theorem 4.1.1, Corollary 4.2.3]{kkko-mon}
For given square-irreducible $L \in \irr_{\mathcal{D}}$ and any $M\in \irr_{\mathcal{D}}$, let $N\in \girr_{\mathcal{D}}$ be the simple head of $L\circ M$.

Then,
\[
\Lambda(L, N) = \Lambda( L, L\circ M)= \Lambda(L, M)
\]
holds, and for any simple subquotient $S\langle s\rangle$ of $\ker (L\circ M \to N)$, with $S\in \irr_{\mathcal{D}}$, we have
\[
\Lambda(L,S) < \Lambda(L,M)\;\mbox{ and } -s < \widetilde{\Lambda}(L,M).
\]
\end{lemma}

\begin{proposition}\label{prop:mult-one}
Let $(L_1,\ldots, L_k)$ be a normal sequence, and $H\langle h\rangle$ be the simple head of the product $L_1\circ \cdots \circ L_k$, with $H\in \irr_{\mathcal{D}}$.

Then, we have
\[
m(L_1\circ \cdots \circ L_k, H)(q)=q^h\;.
\]
In other words, $H$ appears only once in the Jordan-H\"{o}lder series of $(L_1\circ \cdots \circ L_k)^{\fgt}$.

Furthermore, for any $H\not\cong L\in \irr_{\mathcal{D}}$, the powers appearing in the Laurent polynomial $m(L_1\circ \cdots \circ L_k, L)(q)$ are strictly greater than $h$.

\end{proposition}

\begin{proof}
We prove by induction on the product length $k$.

Clearly, $(L_2, \ldots, L_k)$ is a normal sequence. Hence, by Proposition \ref{prop:simplehd}, $L_2\circ \cdots \circ L_k$ has a simple head $H'\langle h'\rangle$, with $H'\in \irr_{\mathcal{D}}$ and $h'\in \mathbb{Z}$. 


By Proposition \ref{prop:prop-kkko}\eqref{it-kk1} and the condition in Proposition \ref{prop:cond-equiv}\eqref{itnorm2}, we see that
\begin{equation}\label{eq:mult1}
\Lambda(L_1, H') = \Lambda(L_1, L_2\circ \cdots \circ L_k )\;.
\end{equation}

Since $L_1$ is square-irreducible, $H\langle h\rangle$ must be the unique simple quotient of $L_1\circ H'\langle h'\rangle$, and $h= h' -\widetilde{\Lambda}(L_1,H')$.


Let $T_1 \subset T_2\subsetneq L_1\circ \cdots \circ L_k$ be submodules, so that $T_2/T_1 \cong L\langle \ell \rangle$ for $L\in \irr_{\mathcal{D}}$. It is enough to show that $L\not\cong H$, and that $\ell > h$.

In case we have an inclusion
\[
\ker(L_1\circ \cdots \circ L_k \to  L_1\circ H'\langle  h' \rangle) \subset T_1\;,
\]
it follows that $L\langle \ell\rangle$ appears as a subquotient of $\ker( L_1\circ H'\langle  h' \rangle \to H\langle h\rangle)$. By Lemma \ref{lem:kkkosqr}, that means $\ell > h$ and $\Lambda(L_1, L) < \Lambda(L_1, H)$, proving $L\not\cong H$.

Otherwise, by exactness of the convolution product, there must exist a subquotient module $L'\langle \ell'\rangle$ of $\ker( L_2\circ \cdots \circ L_k \to H'\langle h'\rangle)$, with $L'\in \irr_{\mathcal{D}}$ and the power $\ell - \ell'$ appearing in the Laurent polynomial $m(L_1\circ L', L)$.

The induction hypothesis now implies that $L'\not\cong H'$ and that $\ell' > h'$.

By Proposition \ref{prop:prop-kkko}\eqref{it-kk2} and \eqref{eq:mult1}, we see that $\Lambda(L_1,L')\leq \Lambda(L_1,H')$, which by the identity \eqref{eq:linrel} implies $\widetilde{\Lambda}(L_1,L')\leq \widetilde{\Lambda}(L_1,H')$. Now, by Lemma \ref{lem:kkkosqr},
\[
\ell - h' > \ell - \ell' \geq -\widetilde{\Lambda}(L_1, L')\geq -\widetilde{\Lambda}(L_1,H')\;,
\]
and the inequality $\ell>h$ follows.

In case that $L$ does not appear as a quotient of $(L_1\circ L')^{\fgt}$, it follows again from Lemma \ref{lem:kkkosqr} that
\[
\Lambda(L_1, L) < \Lambda(L_1, L') \leq \Lambda(L_1,H) = \Lambda(L_1,H)\;,
\]
and $L\not\cong H$ is seen to hold again.

Finally, we are left with the case that $L\langle \ell\rangle$ is a quotient of $L_1\circ L'\langle \ell'\rangle$. The injectivity (\cite[Corollary 3.7]{kkko0}) of the map $N \mapsto$ head$(L_1\circ N)$ on isomorphism classes in $\irr_{\mathcal{D}}$ now implies $L\not\cong H$.

\end{proof}

\begin{remark}
The bound on subquotient shift degrees in the statement of Lemma \ref{lem:kkkosqr} was proved using the geometric insight (see \cite[Lemma 7.5]{McNm},\cite{kato-ext}) coming from a realization of quiver Hecke algebras as extension algebras of perverse sheaves (\`{a} la \cite{vv}). We are unaware of an alternative purely algebraic argument.

Hence, we note that parts of Proposition \ref{prop:mult-one}, though not the simplicity of heads, covertly rely on the underlying geometry of quiver varieties.

\end{remark}

\subsection{Applications for RSK-standard modules}

\begin{theorem}\label{thm:mmain}
  Let $L_\m\in \irr_{\mathcal{D}}$ be a simple module, with $\m \in \Mult$. Let
  \[
  \RSK(\m) = (\la_1,\ldots,\la_{\omega}) \in \Lad^{\omega}
  \]
  be its RSK-transform, and
  \[
\Gamma(\m)=L_{\la_1}\circ \cdots \circ L_{\la_{\omega}}\langle -d(\m)\rangle \in R(\wt(\m))-\gmod
  \]
  its associated RSK-standard module.

Then, $L_{\m}$ is the head (that is, the unique irreducible quotient) of $\Gamma(\m)$, and it appears only once in the Jordan-H\"{o}lder series of $\Gamma(\m)^{\fgt}$ (i.e. $m(\Gamma(\m),L_\m)(q) = 1$).

Moreover, for any $L_{\m}\not\cong L\in \irr_{\mathcal{D}}$, the graded multiplicity $m(\Gamma(\m), L)(q)$, if non-zero, is a polynomial with zero constant term.

\end{theorem}

\begin{proof}

By Propositions \ref{prop:simplehd},\ref{prop:mult-one} and Theorem \ref{thm:gurlap}, it is enough to show that $(L_{\la_1}, \ldots, L_{\la_{\omega}})$ is a normal sequence.

Let $\m'\in \Mult$ be the multisegment, for which $\Vien(\m) = (\la_1, \m')$. Then, $\RSK(\m') = (\la_2,\ldots, \la_{\omega})$.

Arguing inductively by Proposition \ref{prop:cond-equiv}, it is enough to show the equality
\begin{equation}\label{eq:needed}
\widetilde{\Lambda}(\la_1, \m') = \widetilde{\Lambda}(\la_1, \la_2) + \ldots + \widetilde{\Lambda}(\la_1, \la_\omega)\;.
\end{equation}
Note, that the map $\n \mapsto C(\la_1,\n)$ from $\Mult$ to $\mathbb{Z}$ is clearly additively and depends only on $\mathfrak{e}(\n)$. Hence, from Proposition \ref{eq:preser}, we see that
\begin{equation}\label{eq:Cadd}
C(\la_1,\m') =C(\la_1, \la_2 + \ldots + \la_\omega) = C(\la_1, \la_2) + \ldots + C(\la_1, \la_\omega)\;.
\end{equation}
Recall that $(\la_1,\la_i)\in \pairs$, for all $2\leq i\leq \omega$ (\cite[Proposition 2.4]{gur-lap}).
Thus, equation \eqref{eq:needed} follows by substituting the formula of Corollary \ref{cor:maincor} and applying \eqref{eq:Cadd} together with Proposition \ref{eq:preser} once more.

\end{proof}

\begin{theorem}\label{thm:deg}
For $0\neq \m\in \Mult$, the shift of degree defining $\Gamma(\m)$ is explicitly read from the data $\RSK(\m)$ through the formula
\[
d(\m) = -\sum_{j\geq 1} \lambda_j^\ast (\lambda_j^\ast-1) + \#\{(i,j)\;:\; d_{i,j} = c_{i',j'},\mbox{ for }(i',j')\mbox{ with } i<i'\} \;,
\]
where $(P_\m, Q_\m) = ((c_{i,j}), (d_{i,j}))$ is the pair of Young tableaux describing $\RSK(\m)$, and $(\lambda^\ast_1, \lambda^\ast_2, \ldots)$ is the conjugate partition to $\lambda(\m)$ (that is, the column shape of $P_\m, Q_m$).
\end{theorem}
\begin{proof}

Since $\Gamma(\m)$ is defined by a normal sequence (as described in the proof of Theorem \ref{thm:mmain}), we can use Corollary \ref{cor:sumform} to reach an expression for $d(\m)$. Indeed, substituting the formula in Corollary \ref{cor:maincor} gives
\begin{equation}\label{eq:auxeq}
d(\m) = -\sum_{i=2}^\omega (i-1)|\la_i| + \sum_{1\leq i < j \leq \omega} C(\la_i,\la_j)\;,
\end{equation}
where $ \RSK(\m) = (\la_1,\ldots,\la_{\omega})$.

Since $\lambda(\m) = (|\la_1|,\ldots, |\la_\omega|)$, the first part of the sum above may be dually expressed as
\[
-\sum_{j\geq 1} \lambda_j^\ast (\lambda_j^\ast-1)
\]
by standard counting arguments.

Equality between the second parts of the formula appearing in the statement and of equation \eqref{eq:auxeq} follows directly from the definition of the terms involved.

\end{proof}

Transferring the statement of Theorem \ref{thm:mmain} into the $p$-adic setting using the equivalences of Section \ref{sec:equiv-cat}, settles the natural conjectures posed in \cite{gur-lap} regarding the $p$-adic version of RSK-standard modules.

\begin{corollary}\label{cor:final}
Let $\rho\in \cusp_m$ be a supercuspidal representation, and $\pi = Z(\m)\in \irr_{\rho}^{\mathbb{Z}}$ any irreducible representation, given by a multisegment $0\neq \m \in \Mult$ through the Zelevinsky classification.

Let
  \[
  \RSK(\m) = (\la_1,\ldots,\la_{\omega}) \in \Lad^{\omega}
  \]
  be the RSK-transform of $\m$, and
  \[
\widetilde{\Gamma}(\m):= Z(\la_1)\times \cdots \times Z(\la_{\omega}) \in \mathcal{C}^{\mathbb{Z}}_{\rho}
  \]
  its associated RSK-standard module.

Then, $\widetilde{\Gamma}(\m)$ has a unique irreducible representation described by $Z(\m)$. The isomorphism class $Z(\m)$ appear only once in the Jordan-H\"{o}lder series of the representation $\widetilde{\Gamma}(\m)$.

\end{corollary}
\begin{proof}
By Theorem \ref{thm:fullequi} and Proposition \ref{prop:monoid}, $\mathcal{E}_\rho(\widetilde{\Gamma}(\m)) \cong \Gamma(\m)^{\fgt}$. The result then follows from Theorem \ref{thm:mmain} and the description of morphism spaces as in \eqref{eq:homsp}.

\end{proof}

\bibliographystyle{alpha}
\bibliography{propo2}{}

\end{document}